\documentclass[12pt,a4paper,reqno]{amsart}

\usepackage{amsmath, amsfonts, amssymb, amsthm, latexsym, nameref, enumerate} 
\usepackage{amsrefs} 
\usepackage{xspace}
\usepackage{a4wide}
\usepackage{graphicx, subfig}
\usepackage{hyperref}
\usepackage[usenames, dvipsnames]{xcolor}
\usepackage{tikz}
\usetikzlibrary{arrows,decorations.markings,decorations.pathreplacing,plotmarks}
\usepackage{float}
\usepackage{scrextend}
\usepackage{blindtext}

\newtheorem{proposition}{Proposition}[section]
\newtheorem{theorem}[proposition]{Theorem}

\newtheorem{lemma}[proposition]{Lemma}
\newtheorem{thmi}{Theorem}
\newtheorem{cori}[thmi]{Corollary}

\theoremstyle{definition}
\newtheorem{definition}[proposition]{Definition}

\newtheorem{example}[proposition]{Example}
\newtheorem{remark}[proposition]{Remark}

\newcommand{\set}[1]{\left\{#1\right\}}
\newcommand{\setcon}[2]{\left\{#1\ \left|\ #2\right.\right\}}

\newcommand{\abs}[1]{\left\lvert#1\right\rvert}
\newcommand{\into}{\hookrightarrow}

\newcommand{\R}{\mathbb{R}}

\newcommand{\N}{\mathbb{N}}
\newcommand{\eps}{\varepsilon}

\newcommand{\diam}{\textup{diam}}
\newcommand{\Cay}{\textup{Cay}}

\newcommand{\Lab}{\textup{Lab}}

\newcommand{\TCG}{\mathcal{TC}(\mathcal P)}
\newcommand{\PFin}{\mathcal P(\omega)/Fin}

\newcommand{\ngen}[1]{\left\langle\!\left\langle #1 \right\rangle\!\right\rangle}
\newcommand{\fpres}[2]{\left\langle #1 \left| #2 \right.\right\rangle}

\title[Actions of small cancellation groups on hyperbolic spaces]{Actions of small cancellation groups\\on hyperbolic spaces}

\author[Abbott]{Carolyn R. Abbott}
 \email{c\underline{\,\,\,}abbott@math.berkeley.edu}
\address{Department of Mathematics, University of California, Berkeley, Berkeley, CA 94720}

\author[Hume]{David Hume}
 \email{david.hume@maths.ox.ac.uk}
\address{Mathematical Institute, University of Oxford, Woodstock Road, Oxford OX2 6GG}

\date{\today}

\begin{document}

\begin{abstract}
We generalize Gruber--Sisto's construction of the coned--off graph of a small cancellation group to build a partially ordered set $\mathcal{TC}$ of cobounded actions of a given small cancellation group whose smallest element is the action on the Gruber--Sisto coned--off graph. In almost all cases $\mathcal{TC}$ is incredibly rich: it has a largest element if and only if it has exactly 1 element, and given any two distinct comparable actions $[G\curvearrowright X] \preceq [G\curvearrowright Y]$ in this poset, there is an embeddeding $\iota:P(\omega)\to\mathcal{TC}$ such that $\iota(\emptyset)=[G\curvearrowright X]$ and $\iota(\N)=[G\curvearrowright Y]$. We use this poset to prove that there are uncountably many quasi--isometry classes of finitely generated group which admit two cobounded acylindrical actions on hyperbolic spaces such that there is no action on a hyperbolic space which is larger than both.
\end{abstract}

\maketitle

\section{Introduction}
The study of acylindrical actions on hyperbolic spaces is a powerful tool for understanding algebraic properties of groups that admit aspects of non-positive curvature. The class of groups that admit such actions on non-elementary hyperbolic spaces, called acylindrically hyperbolic groups, is incredibly rich, including non-elementary hyperbolic and relatively hyperbolic groups, non-exceptional mapping class groups, $\operatorname{Out}(\mathbb F_n)$ for $n\geq 2$, and non-directly decomposable, non-virtually cyclic right-angled Artin and Coxeter groups, among many others.  Moreover, the consequences of being acylindrically hyperbolic are far-reaching.  Such groups are SQ-universal, have non-abelian free normal subgroups, a maximal finite normal subgroup, infinite dimensional second bounded cohomology, and a well-developed small cancellation theory \cite{DGO17, Hull, BestvinaFujiwara, Hamenstadt}.  

A single acylindrically hyperbolic group will admit many different acylindrical actions on different hyperbolic spaces, and it is natural to ask how these actions relate to each other.  This kind of question was made precise in  \cite{AHO17}, where the authors and Osin define a partial order on the set of actions of a group on a metric space as follows: $G\curvearrowright X \preceq G\curvearrowright Y$ if given any points $x\in X$, $y\in Y$, the map $(G.y,d_Y)\to (G.x,d_X)$ given by $g.y\mapsto g.x$ is coarsely Lipschitz\footnote{Strictly speaking the partial order is on the set of classes of equivalent actions.}.  The largest action of a group in this partial ordering is always the action on its Cayley graph and the smallest action is the action  on a point.

Under this partial ordering, the set of (equivalence classes of) cobounded actions of a given group $G$ on metric spaces  has a natural poset structure; we call this poset $\mathcal A_{cb}(G)$.\cite{ABO17}  Moreover, it is shown in \cite{ABO17} that $\mathcal A_{cb}(G)$ is isomorphic (as a poset) to the set of (possibly) infinite generating sets of $G$, which we call $\mathcal G(G)$.

Let $\mathcal H(G)\subset \mathcal G(G)$ be the set of equivalence classes $[X]$ of generating sets of $G$ such that $\Gamma(G,X)$ is hyperbolic for some (equivalently, any) representative $X$ of $[X]$, and we let $\mathcal {AH}(G)\subset \mathcal H(G)$ be the set of equivalence classes $[Y]$ of generating sets of $G$ such that $\Gamma(G,Y)$ is hyperbolic and the action $G\curvearrowright \Gamma(G,Y)$ is acylindrical.  

\subsection{Small cancellation groups}
In this paper, we investigate the structure of $\mathcal H(G)$ and $\mathcal{AH}(G)$ for the class of small cancellation groups.  Small cancellation theory provides a rich class of finitely generated groups which can be constructed to satisfy rather exotic properties.  Graphical small cancellation theory, a generalization of classical small cancellation introduced by Gromov, is a tool that allows one to construct groups whose Cayley graphs have  prescribed subgraphs.  In \cite{GruberSisto}, it is shown that all infinitely presented $Gr'(\frac16)$ graphical small cancellation groups are acylindrically hyperbolic.   Thus it is natural to look for hyperbolic spaces on which such groups act acylindrically.  We describe an uncountable family of such spaces as a subset of $\mathcal G(G)$.  

Let $\mathcal P=\fpres{S}{r_1,r_2,\ldots}$ be a presentation defining a group $G$ where each $r_i$ is cyclically reduced, and let $\overline R$ be the set of all cyclic conjugates of the $r_i$ and their inverses. Roughly, a piece in $r\in\overline R$ is a subword of $r$ that also appears as a subword of some distinct $r'\in\overline R$. Let $L$ be the union of $S$ and the set of all initial subwords of all $r\in \overline R$, and let $P^4$ be the set of all words in $G$ which are a product of at most 4 pieces. Let $\mathcal G^4_L(\mathcal P)$ be the set of equivalences classes $[X]$ of generating sets with a representative $P^4\subseteq X\subseteq L$. For each $r_i$ let $C_i$ be a cycle labelled by $r_i$, and let $X_i$ be the subset of $X$ consisting of subwords of the cyclic conjugates of $r_i$.  For each $i$ and each $x\in X_i^{\pm 1}$, add an edge to $C_i$ between the initial and terminal vertex of any subpath of $C_i$ labeled by $x$. By doing so, for each $i$, we get a new graph which we call $C_i^X$.

\begin{definition} The {\bf poset of thin cones} $\TCG$ is the subset of all $X\in\mathcal G^4_L(\mathcal P)$ with the property that there exists a constant $\delta\geq 0$ such that, for every $i$, $C_i^X$ is $\delta$-hyperbolic. 
\end{definition}

Given $\lambda>0$, we say $\mathcal P$ satisfies the $C'(\lambda)$ small cancellation condition (or just $\mathcal P$ is $C'(\lambda)$) if the length of any piece in $r$ is no longer that $\lambda$ times the length of $r$.

It was shown by Gruber and Sisto in \cite{GruberSisto} that whenever $\mathcal P$ satisfies $C'(\frac16)$, then $[L]\in\mathcal H(G)$.  Moreover, Coulon and Gruber show in \cite{Coulon-Gruber} that if $\mathcal P$ is {\bf uniformly power-free}, that is, there is some $n$ such that $x^n\not\in L$ for all $x\in F(S)\setminus\set{1}$, then $[L]\in\mathcal{AH}(G)$, and every infinite order element of $G$ is loxodromic with respect to the action of $G$ on $\Cay(G,L)$\footnote{In fact both of these results are proved in the much more general setting of graphical small cancellation theory}.  We will show that under a slightly stronger hypothesis, the same results can be obtained for all elements of $\TCG$.

\begin{thmi} \label{thm:hypandacyl}  Let $\mathcal P=\fpres{S}{r_1,r_2,\dots}$ be a $C'(\frac{1}{24})$ presentation defining a group $G$.  Then $\TCG\subset \mathcal H(G)$.  Moreover, if $G$ is uniformly power-free, then $\TCG\subset \mathcal{AH}(G)$.
\end{thmi}

By construction, $[L]$ is the smallest element in $\TCG$, i.e., $[L]$ is comparable to and smaller than every other element of $\TCG$.  

We note that our thin cones construction starts with the Gruber-Sisto action and then builds \emph{larger} actions.  This is in contrast to previous constructions, as in \cite{DGO17} and \cite{ABO17}, which typically start with a given action and produce smaller actions.  The main difficulty in producing larger actions is in managing to construct spaces for the actions which are hyperbolic.

We next describe the structure of $\TCG$.  Recall that $\mathcal P(\omega)/Fin$ is the poset of equivalence classes of subsets of $\mathbb N$, where two subsets $A,B\subseteq \mathbb N$ are equivalent if $|A\triangle B|<\infty$ and $A\leq B$ if $|A\setminus B|<\infty$.  We note that $\mathcal P(\omega)/Fin$ contains a copy of $\mathcal P(\omega)$, as follows.  Write $\N$ as a union of infinitely many infinite subsets $A_1,A_2,\ldots$.  Then the set of all subsets of $\N$ equal to a union of $A_i$'s is an embedded copy of $\mathcal P(\omega)$.  

\begin{thmi}\label{thmi:TCG}
Let $\mathcal P=\fpres{S}{r_1,\ldots}$ be a $C'(\frac{1}{24})$ presentation of a group $G$.  If $\mathcal P$ is uniformly power-free, then $\TCG\subseteq \mathcal{AH}(G)$.  Moreover, if $\mathcal P$ is power-free but not uniformly so (for every $x\in F(S)\setminus\set{1}$ there exists an $n$ such that $x^n\not\in L$), then $\TCG\subseteq \mathcal H(G)\setminus \mathcal{AH}(G)$.   Additionally, one of the following occurs.
\begin{enumerate}
\item $|\TCG|=1$, which occurs if and only if each $C_i^{\mathcal P^4}$ has bounded diameter, or, equivalently, each $r_i$ is a product of a uniformly bounded number of pieces.

\item $|\TCG|=2^{\aleph_0}$, and $\TCG$ has the following structure:
	\begin{itemize}
	\item There exist $[X],[Y]\in\TCG$ such that there is no $[Z]\in\TCG$ satisfying $[X]\preceq [Z]$ and $[Y]\preceq [Z]$.
	\item For every distinct pair $[X],[Y]\in\TCG$ such that $[X]\preceq[Y]$, there is an embedding of posets $\mathcal P(\omega)/Fin\into \TCG$ such that for each $[Z]\in\mathcal P(\omega)/Fin$, $[X]\preceq [Z] \preceq [Y]$.
	
	\item Every $[X]\in\TCG$ which is not the minimal element is contained in an uncountable chain and in an uncountable antichain in $\TCG$.
	
	\end{itemize}

\end{enumerate}
\end{thmi}
The first point of (ii) is particularly striking, since in many natural examples  $\mathcal{AH}(G)$ is known to have largest elements. Our next goal is use $\TCG$ to study the larger posets $\mathcal H(G)$ and $\mathcal{AH}(G)$.

\subsection{Accessibility}

 A largest element in $\mathcal{H}(G)$ (respectively, $\mathcal{AH}(G)$) corresponds to a ``best" action (respectively, acylindrical action) of $G$ on a hyperbolic space.  If a largest element in $\mathcal H(G)$ (respectively, $\mathcal{AH}(G)$)  exists, we say the group is {\bf $\mathcal H$--accessible} (respectively, {\bf $\mathcal{AH}$--accessible}).  Notice that if a group is not hyperbolic, then the action on its Cayley graph will not be an element of either poset, and thus the largest element, if it exists, will not be a proper cocompact action.  All hyperbolic groups (and all their finitely generated subgroups), mapping class groups, fundamental groups of 3-manifolds, and a class of CAT(0) cubical groups which includes all virtually special groups are $\mathcal{AH}$--accessible \cite{ABO17,ABD17}.

One obstruction to $\mathcal{AH}$--accessibility can be found by considering the set of loxodromic elements in the different acylindrical actions.   An acylindrical action of a group on a hyperbolic space in which every  element that is loxodromic in some acylindrical action on a hyperbolic space is loxodromic in this action is called a {\bf universal acylindrical action}. 

A group which does not admit a universal acylindrical action cannot be $\mathcal{AH}$--accessible. The first author used this obstruction in \cite{Abbott_Dunwoody} to show that Dunwoody's inaccessible group, which is finitely generated but infinitely presented, is not $\mathcal{AH}$--accessible.  Moreover, in \cite{AH18}, the authors construct $2^{\aleph_0}$ quasi-isometry classes of torsion-free groups which do not admit universal acylindrical actions.  However, this is not the only obstruction to $\mathcal{AH}$--accessibility.  In \cite{ABO17}, an example is given of a finitely presented group which admits a universal acylindrical action on a hyperbolic space but is neither $\mathcal H$-- nor $\mathcal{AH}$--accessible.  

By construction, the Gruber-Sisto action is a universal acylindrical action, which implies that every element in $\TCG$ is a universal acylindrical action, as well.

\begin{definition}
A group $G$ is \emph{weakly $\mathcal{H}$--accessible} (respectively, \emph{weakly $\mathcal{AH}$--accessible}) if there  exists an action (respectively, acylindrical action) of $G$ on a hyperbolic space $X$ such that $G\curvearrowright \Gamma(G,Y)\preceq G\curvearrowright X$ for all actions $[Y]\in\mathcal{H}(G)$ (respectively, all actions $[Y]\in\mathcal{AH}(G)$).  We do not require that $G\curvearrowright X$ corresponds to an element of $\mathcal{H(G)}$ (respectively, $\mathcal{AH}(G)$), that is, the action may not be cobounded.
\end{definition}

Clearly, if $G$ is $\mathcal{H}$--accessible, then it is weakly $\mathcal{H}$--accessible, and similarly, $\mathcal{AH}$--accessibility implies weak $\mathcal{AH}$--accessibility.

\begin{thmi} \label{thmi:notwklyaccess}
If $\mathcal{P}=\fpres{S}{R}$ is a presentation of a group $G$ which satisfies $C'(\frac{1}{24})$ and has ``enough pieces" (to be defined precisely later), then there is an uncountable set $\mathcal U\subset \TCG$ such that for any two elements $[X_1],[X_2]\in U$, if a (not necessarily cobounded) action $G\curvearrowright Y$ dominates $G\curvearrowright\Gamma(G,X_i)$ for $i=1,2$, then $Y$ is not hyperbolic.  In particular, $G$ is not weakly $\mathcal H$--accessible. 
\end{thmi}

The above theorem gives the first examples of groups which are not weakly $\mathcal H$--accessible.  We also have the following immediate corollary.

\begin{cori}
There are $2^{\aleph_0}$ quasi-isometry classes of finitely generated groups which admit a universal acylindrical action on a hyperbolic space, but are neither $\mathcal{H}$-- nor $\mathcal{AH}$--accessible.
\end{cori}

\subsection*{Small cancellation constants} We have no reason to believe that our results cannot be improved to $C'(\frac16)$ presentations. The current method can be used to prove all of the above results in the $C'(\frac{1}{14})$ setting (which is the largest $\lambda$ such that Proposition \ref{prop:combgeods} holds). Any improvement beyond this seems to require a different approach, as this is the first time we are able to apply known classifications of polygons in small cancellation groups. We choose to work in the $C'(\frac{1}{24})$ setting as it is still technically difficult, and presents many of the same challenges as the $C'(\frac{1}{14})$ setting, but avoids 25 additional pages of agonising case-by-case proofs. It is worth noting that one can further reduce the difficulty of the arguments by working with $C'(\frac{1}{30})$ presentations.

\subsection*{Plan of the paper}
In section \ref{sec:conedoffgraph} we give the construction of the new graphs and explain how to relate paths in the Cayley graph to geodesics in our new graph.  In section \ref{sec:hyperbolicity}, we gather various properties of bigons in our new graphs and prove that our new graphs are hyperbolic. In section \ref{sec:acylindricity}, gather various properties of quadrangles in our new graphs and show that under certain hypotheses the action of the group on the new graph is acylindrical, which proves Theorem \ref{thm:hypandacyl}.  Finally we use the actions on these new graphs to prove Theorems \ref{thmi:TCG} and \ref{thmi:notwklyaccess}
 in section \ref{sec:largest}.

\subsection*{Acknowledgements}
The authors are grateful to R\'{e}mi Coulon and Dominik Gruber for interesting conversations, and for sharing with us the results of their paper. The first author was partially supported by the NSF RTG awards DMS-1502553. The second author was supported by the NSF grant DMS-1440140 while the author was in residence at the Mathematical Sciences Research Institute in Berkeley, California, during the Fall 2016 semester, and by a Titchmarsh Research Fellowship from the University of Oxford.

\section{Preliminaries}
\subsection{Hyperbolicity} A geodesic metric space $X$ is $\delta$--hyperbolic if, for every geodesic triangle with sides $\gamma_1,\gamma_2,\gamma_3$ we have
\begin{equation}\label{spacehyp}
 \gamma_1 \subseteq N_{\delta}(\gamma_2\cup\gamma_3):= \setcon{x\in X}{d_X(x,\gamma_2\cup\gamma_3)\leq\delta}.
\end{equation}
We say a geodesic metric space is hyperbolic if it is $\delta$--hyperbolic for some $\delta$.

For a graph $\Gamma$ equipped with the shortest path metric, $\Gamma$ is $\delta$--hyperbolic for some $\delta$ if and only if there exists some $\delta'$ such that for every geodesic bigon $\gamma_1,\gamma_2$ we have
\begin{equation}\label{graphhyp}
 \gamma_1 \subseteq N_{\delta'}(\gamma_2),
\end{equation}
by \cite[Theorem 2]{Papas_bigons}.

\subsection{Acylindrically hyperbolic groups}

An action of a group $G$ by isometries on a metric space $X$ is \textbf{acylindrical} if for all $\varepsilon>0$ there exist constants $M,N\geq 0$ such that for all $x,y\in X$ with $d(x,y)\geq M$, the number of elements $g\in G$ satisfying $d(x,gx)\leq \varepsilon$ and $d(y,gy)\leq \varepsilon$ is at most $N$. Recall that given a group $G$ acting on a hyperbolic metric space $X$, an element $g\in G$ is {\bf loxodromic} if the map $\mathbb Z\to X$ defined by $n\mapsto g^nx$ is a quasi-isometric embedding for some (equivalently any) $x\in X$.  However, an element of $G$ may be loxodromic for some actions and not for others.  
Consider, for example, the free group on two generators acting on its Cayley graph and acting on the Bass-Serre tree associated to the splitting $\mathbb F_2\simeq \langle x\rangle *\langle y\rangle$.  In the former action, every non-trivial element is loxodromic, while in the latter action, no powers of $x$ and $y$ are loxodromic.  An element $g$ of a group $G$ is a \textbf{generalized loxodromic} if there is an acylindrical action of $G$ on a hyperbolic space $X$ such that $g$ acts loxodromically. A non-virtually cyclic group is \textbf{acylindrically hyperbolic} if and only if it contains a generalized loxodromic element \cite{Os16}.  An acylindrical action of a group on a hyperbolic space is a {\bf universal acylindrical action} if every generalized loxodromic element is loxodromic.

The following notions are discussed in detail in \cite{ABO17}.  We give a brief overview here.
 Fix a group $G$.  Given a (possibly infinite) generating set $X$ of $G$, let $|\cdot|_X$
 denote the word metric with respect to $X$.  Given two generating
 sets $X$ and $Y$, we say $X$ is {\bf dominated} by $Y$ and write
 $X\preceq Y$ if
\[\sup_{y\in Y}|y|_X<\infty.\] Note that when 
$X\preceq Y$, then the action $G\curvearrowright
\Gamma(G,Y)$ provides more information about the group than
$G\curvearrowright\Gamma(G,X)$, and so, in some sense, is a ``larger"
action.
The two generating sets $X$ and $Y$ are equivalent if $X\preceq Y$ and
$Y\preceq X$; when this happens we write $X\sim Y$.  

We let $\mathcal G(G)$ be the set of all equivalence classes of generating sets of $G$ and let $\mathcal H(G)$ (respectively, $\mathcal{AH}(G)$) be the set of equivalence classes of generating
sets $X$ of $G$ such that $\Gamma(G,X)$ is hyperbolic (respectively, $\Gamma(G,X)$ is hyperbolic and the action
$G\curvearrowright \Gamma(G,X)$ is acylindrical), where $\Gamma(G,X)$ 
is the Cayley graph of $\Gamma$ with respect to the generating set $X$.  We denote the equivalence class of $X$ by $[X]$.  The preorder
$\preceq$ induces an order on $\mathcal{AH}(G)$, which we also denote 
$\preceq$.   We say an equivalence class of generating sets is {\bf largest} if it is the largest
element in $\mathcal{(A)H}(G)$ under this ordering.

Let $\mathcal A_{cb}(G)$ be the set all equivalence classes of cobounded $G$--actions on geodesic metric spaces.  Given a cobounded action of $G$ on a geodesic metric space $S$, a Svarc-Milnor argument gives a (possibly infinite) generating set $Y$ of $G$ such that there is a $G$--equivariant quasi-isometry between $S$ and $\Gamma(G,Y)$.  We define
\begin{equation}
\sigma\colon \mathcal A_{cb}(G)\to \mathcal G(G)\end{equation} to be the map sending $[G\curvearrowright S]$ to $[Y]$, which is an isomorphism of posets by \cite[Proposition 3.12]{ABO17}.  It is clear that if $G\curvearrowright S$ is an action (respectively, acylindrical action) on a hyperbolic space, then $[Y]\in\mathcal H(G)$ (respectively, $\mathcal{AH}(G)$).

\subsection{Small cancellation theory} Given a group $G$ which is generated by a symmetric set $X$ we denote the word metric on $G$ with respect to $X$ by $d_X$ and define $\abs{g}_X=d_X(1,g)$ for all $g\in G$.

Given a set $S$, we denote by $\mathcal M(S)$ the free monoid over $S$. We define a formal inversion in $\mathcal M(S\sqcup S^{-1})$ by the rule
\[
 (s_1^{\varepsilon_1}s_2^{\varepsilon_2}\dots s_n^{\varepsilon_n})^{-1} = s_n^{-\varepsilon_n}\dots s_2^{-\varepsilon_2} s_1^{-\varepsilon_1},
\]
where we associate each $s\in S$ with $s^{+1}$. 

Let $S$ be a set, and let $F(S)$  denote the free group freely generated by $S$. Let $R$ be a set of cyclically reduced elements of $F(S)$ (that is, each $r\in R$ is of minimal length in its conjugacy class), and define $\overline{R}$ to be the closure of $R$ under reduced cyclic conjugation and inversion. A word $u\in F(S)$ is an \textbf{initial subword} of a $r\in F(S)$ if there exists some $t\in F(S)$ such that $r=ut$ is a reduced decomposition of $r$, i.e. the equality holds in $F(S)$ and $\abs{r}_S=\abs{u}_S+\abs{t}_S$. A \textbf{piece} of $R$ is a word $u\in F(S)$ which is an initial subword of at least two distinct elements of $\overline{R}$. Given $\lambda>0$, we say that the presentation $\fpres{S}{R}$ satisfies the $C'(\lambda)$ \textbf{small cancellation condition} if for any piece $u$ which is an initial subword of $r\in\overline{R}$ we have
\begin{equation} 
 \abs{u}_S < \lambda \abs{r}_S.
\end{equation}
A group $G$ is called a $C'(\lambda)$ group if it admits a presentation $\fpres{S}{R}$ which satisfies the $C'(\lambda)$ small cancellation condition. We will not assume that $S$ is finite in general.

\subsection{Diagrams}

\begin{definition}[Diagram]
  A \textbf{diagram} is a finite, simply-connected,
  2--dimensional CW complex with an embedding into the plane,
  considered up to orientation-preserving homeomorphisms of the plane.
A diagram is called a \textbf{disc diagram} if it is homeomorphic to a disc.
\end{definition}

An \textbf{arc} in a diagram $D$ is a maximal path of length at least 1 all of whose interior vertices have valence 2 in $D$.
An \textbf{interior arc} is an arc whose interior is contained in the
interior of $D$, and an \textbf{exterior arc} is an arc contained in the
boundary of $D$.
A \textbf{face} is the image of a closed 2-cell of $D$.
If $\Pi$ is a face, its \textbf{interior degree} $i(\Pi)$ is the number
of maximal interior arcs in its boundary. 
Likewise, its \textbf{exterior degree} $e(\Pi)$ is the number of
maximal exterior arcs. An \textbf{interior face} is one with exterior degree 0; an \textbf{exterior face} is one with positive exterior degree.

One key result we will use for diagrams is Strebel's curvature formula \cite{Str90}. Let $D$ be a disc diagram without vertices of degree $2$. Then
\begin{equation}\label{Strebcurv}
6=2\sum_v (3-d(v)) +\sum_{e(B)=0} (6-i(B)) +\sum_{e(B)=1}(4-i(B)) +\sum_{e(B)\geq 2}(6-2e(B)-i(B)).
\end{equation}

One obvious consequence of this is the following special case of Greendlinger's lemma.

\begin{lemma}\label{lem:Greendl} Let $D$ be a disc diagram. There is a face $B\subseteq D$ whose boundary consists of 1 exterior arc and at most 3 interior arcs.
\end{lemma}
\begin{proof} Replace all interior arcs and exterior arcs by edges, so that the resulting diagram has no vertices of degree $\leq 2$. By (\ref{Strebcurv}) the only possible positive contribution from the right hand side is from a face satisfying $e(B)=1$ and $i(B)\leq 3$. Since the left hand side is positive there must be such a face.
\end{proof}

We will use the classifications of certain diagrams very heavily throughout the paper, and so we recall the definitions and main results here.

\begin{definition}
  A \textbf{$(3,7)$--diagram} is a diagram such that every interior vertex
  has valence at least three and every interior face has interior
  degree at least seven.
\end{definition}

\begin{definition}\label{def:combinatorialgeodesicpolygon}
  A \textbf{combinatorial geodesic $n$--gon} $(D,(\gamma_i)_i)$ is a
$(3,7)$--diagram $D$ whose boundary is a concatenation of immersed subpaths
$\gamma_0,\dots,\gamma_{n-1}$ (called \textbf{sides}) such that each boundary face whose exterior part is a single arc that is contained in one of the
  $\gamma_i$ has interior degree at least 4.
A valence 2 vertex that belongs to more than one side is called a
\textbf{distinguished vertex}.
A face whose exterior part contains an arc not contained in one of the
sides is a \textbf{distinguished face}. A combinatorial geodesic $n$--gon is \textbf{simple} if its boundary is a simple cycle, and \textbf{non--degenerate} if the same diagram cannot be expressed as a combinatorial geodesic $k$--gon for any $k<n$.
\end{definition}
We use the terms bigon, triangle and quadrangle in place of $2$--, $3$-- and $4$--gon respectively. 

\begin{theorem}[Strebel's classification]\label{thm:Strebel} A simple combinatorial geodesic bigon has the form $I_1$ below, a simple non--degenerate combinatorial triangle has one of the forms $I_2$, $I_3$, $II$, $III_1$, $IV$ or $V$:
\end{theorem}
\begin{figure}[H]
\centering
\begin{tikzpicture}[yscale=1,xscale=1, 
vertex/.style={draw,fill,circle,inner sep=0.3mm}]

\draw[thin] (0,3) .. controls (-0.5,2.33) and (-0.5,1.66) .. (0,1) .. controls (0.5,1.66) and (0.5,2.33) .. (0,3);

\path (0,1) node[below] {I${}_1$};

\begin{scope}
\clip (0,3) .. controls (-0.5,2.33) and (-0.5,1.66) .. (0,1) .. controls (0.5,1.66) and (0.5,2.33) .. (0,3);

\draw[thick, dotted]
				(-2,2.75) -- (2,2.75)  
				(-2,2.5) -- (2,2.5)
				(-2,1.5) -- (2,1.5)
				(-2,1.25) -- (2,1.25);
\end{scope}

\draw[xshift=4.5cm, thin] (-0.3,3) .. controls (-0.5,2.33) and (-0.5,1.66) .. (0,1) .. controls (0.5,1.66) and (0.5,2.33) .. (0.3,3) -- (-0.3,3);

\path (4.5,1) node[below] {I${}_2$};

\begin{scope}[xshift=4.5cm]
\clip (-0.3,3) .. controls (-0.5,2.33) and (-0.5,1.66) .. (0,1) .. controls (0.5,1.66) and (0.5,2.33) .. (0.3,3) -- (-0.3,3);

\draw[thick, dotted]
				(-2,2.75) -- (2,2.75)  
				(-2,2.5) -- (2,2.5)
				(-2,1.5) -- (2,1.5)
				(-2,1.25) -- (2,1.25);

\end{scope}

\draw[xshift=9cm, thin]    (0,3) -- (-1,1) -- (1,1) -- (0,3);

\path (9,1) node[below] {I${}_3$};

\begin{scope}[xshift=9cm]
\clip (0,3) -- (-1,1) -- (1,1) -- (0,3);

\draw[very thin]    (-1,1.9) -- (0,0.9)
					(1,1.9) -- (0,0.9);

\draw[thick, dotted]		
					(-1,1.4) -- (0,0.4)
					(-1,1.2) -- (0,0.2)
					(1,1.4) -- (0,0.4)
					(1,1.2) -- (0,0.2);

\end{scope}

\draw[thin]    (0,0) -- (-1,-2) -- (1,-2) -- (0,0);

\path (0,-2) node[below] {II};

\begin{scope}
\clip (0,0) -- (-1,-2) -- (1,-2) -- (0,0);

\draw[thick, dotted]
					(-1,-1.6) -- (0,-2.6)
					(-1,-1.8) -- (0,-2.8)
					(1,-1.6) -- (0,-2.6)
					(1,-1.8) -- (0,-2.8)
					(-1,-0.5) -- (1,-0.5)
					(-1,-0.3) -- (1,-0.3);
					
\draw[very thin]
				    (-1,-1.2) -- (0,-2.2)
					(1,-1.2) -- (0,-2.2)
					(-1,-0.8) -- (1,-0.8);
					
\end{scope}

\draw[xshift=3cm, thin]    (0,0) -- (-1,-2) -- (1,-2) -- (0,0);

\path (3,-2) node[below] {III${}_1$};

\begin{scope}[xshift=3cm]
\clip (0,0) -- (-1,-2) -- (1,-2) -- (0,0);

\draw[thick, dotted]
					(-1,-1.6) -- (0,-2.6)
					(-1,-1.8) -- (0,-2.8)
					(1,-1.6) -- (0,-2.6)
					(1,-1.8) -- (0,-2.8)
					(-1,-0.5) -- (1,-0.5)
					(-1,-0.3) -- (1,-0.3);
					
\draw[very thin]
				    (-1,-1) -- (0,-2)
					(1,-1) -- (0,-2)
					(-1,-0.8) -- (1,-0.8);
\end{scope}

\draw[xshift=6cm, thin]    (0,0) -- (-1,-2) -- (1,-2) -- (0,0);

\path (6,-2) node[below] {IV};

\begin{scope}[xshift=6cm]
\clip (0,0) -- (-1,-2) -- (1,-2) -- (0,0);

\draw[very thin]    (-1,-1.2) -- (0,-2.2)
					(1,-1.2) -- (0,-2.2)
					(-1,-0.8) -- (1,-0.8);

\draw[thick, dotted]
					(-1,-1.6) -- (0,-2.6)
					(-1,-1.8) -- (0,-2.8)
					(1,-1.6) -- (0,-2.6)
					(1,-1.8) -- (0,-2.8)
					(-1,-0.5) -- (1,-0.5)
					(-1,-0.3) -- (1,-0.3);
					
	\begin{scope}
		\clip (-1,-1.2) -- (0,-2.2) -- (1,-1.2) -- (1,-0.8) -- (-1,-0.8) -- (-1,-1.2);
		
		\draw[very thin]	(0,-1.4) -- (-1,-2)
							(0,-1.4) -- (1,-2)
							(0,-1.4) -- (0,0);
	\end{scope}
					
\end{scope}

\draw[xshift=9cm, thin]    (0,0) -- (-1,-2) -- (1,-2) -- (0,0);

\path (9,-2) node[below] {V};

\begin{scope}[xshift=9cm]
\clip (0,0) -- (-1,-2) -- (1,-2) -- (0,0);

\draw[very thin]
					(0,-1.6) -- (0,-2)
					(0,-1.6) -- (-1.6,0)
					(0,-1.6) -- (1.6,0);

\draw[thick, dotted]
					(-1,-1.6) -- (0,-2.6)
					(-1,-1.8) -- (0,-2.8)
					(1,-1.6) -- (0,-2.6)
					(1,-1.8) -- (0,-2.8)
					(-1,-0.5) -- (1,-0.5)
					(-1,-0.3) -- (1,-0.3);

\end{scope}

\end{tikzpicture}
\caption{Combinatorial geodesic bigons and triangles. The dotted lines indicate optional additional interior arcs; if they appear, there may be any finite number of them.}
\end{figure}
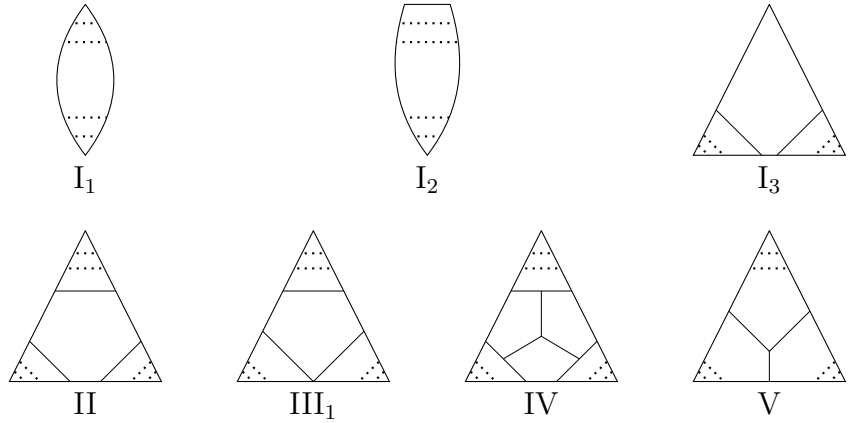

Given a presentation $\fpres{S}{R}$, a diagram over this presentation is a diagram where all edges are oriented and labelled by an element of $S$, and such that the label of the boundary of any $2$-cell is equal (as an element of $F(S)$) to an element of $\overline{R}$. If the edge $e=xy$ is labelled by $s$ and is oriented from $x$ to $y$ then we say that the label of the directed edge $(x,y)$ is $s$ and the label of the directed edge $(y,x)$ is $s^{-1}$. We denote this by $\Lab(x,y)=s$ and $\Lab(y,x)=s^{-1}$. The label of a directed path $P=(x=x_0,\ldots,x_n=y)$ (so each $x_ix_{i+1}$ is an edge) is
\[
 \Lab(p)= \Pi_{i=0}^{n-1} \Lab(x_i,x_{i+1})
\]
considered as a word in the free monoid $\mathcal M(S\sqcup S^{-1})$. Notice that since $\overline{R}$ contains all cyclically reduced conjugates of the elements of $R\cup R^{-1}$ it does not matter which vertex on the boundary we choose to start from or which orientation of the loop we choose, since the resulting words are either all in $\overline{R}$ or all not in $\overline{R}$.

The \textbf{boundary word} of a diagram $D$ over a presentation $\fpres{S}{R}$ is the label of a shortest length closed path $P$ in the $1$--skeleton of $D$ whose image contains $\partial D$. It is well--defined up to formal inversion and cyclic permutation of letters.

Diagrams are the main tool for studying small cancellation groups.  The existence of diagrams with given boundary word is guaranteed by the Van-Kampen lemma.

\begin{lemma}
Let $\fpres{S}{R}$ be a $C'(\frac16)$ presentation, and let $w\in \mathcal M(S\sqcup S^{-1})$ represent the identity in $G=\fpres{S}{R}$. There exists a diagram $D$ over this presentation with boundary word $w$ in which the label of every interior arc is a piece.
\end{lemma}

\section{Coned-off graphs}\label{sec:conedoffgraph}
In this section, we construct new graphs on which small cancellation groups act and gather some properties of these graphs which will be useful for the proofs of hyperbolicity and acylindricity.  
\medskip

\subsection{Constructing new graphs}

Let $\mathcal P=\fpres{S}{R}$ be a $C'(\frac{1}{24})$ presentation, and define $G=F(S)/\ngen{R}$.\footnote{We will implicitly assume throughout that every $s\in S$ appears in some $r\in R$. If $S'$ is the set of $s\in S$ which do not appear in any $r\in R$, then $G = F(S\setminus S')/\ngen{R} \ast F(S')$.  All our results can be applied to $F(S\setminus S')/\ngen{R}$ and immediately lifted to the original presentation.} Enumerate $R=\set{r_1,r_2,\ldots}$. Recall the following sets, which were defined in the introduction:
	\begin{itemize}
	\item $L$ is the union of $S$ and the set of all initial subwords of the cyclic conjugates of $r_i$ and their inverses, 
	\item $P^4$ is the set of all words in $F(S)$ which are a product of at most 4 pieces, 
	\item $\mathcal G^4_L(\mathcal P)$ is the set of equivalences classes $[X]$ of generating sets with a representative $P^4\subseteq X\subseteq L$, and  
	\item  for each $i$, $X_i$ is the subset of $X$ consisting of subwords of the cyclic conjugates of $r_i$ and $r_i^{-1}$.
	\end{itemize}

For any $[X]\in\mathcal G^4_L(\mathcal P)$ and any representative $X$ of $[X]$, $G$ acts on the Cayley graph $\Gamma(G,X)$.  We now give a more geometric description of this Cayley graph for a particular representative $X$ satisfying $P^4\subseteq X\subseteq L$.   

For each $r_i$ let $C_i$ be an oriented cyclic graph of length $\abs{r_i}_S$ whose edges are labeled by elements of $S$ so that the concatenation of these labels (respecting the orientation) is a cyclic conjugate of $r_i$.  Let $X_i$ be the subset of $X$ consisting of subwords of the cyclic conjugates of $r_i$.  For each $i$ and each $x\in X_i^{\pm 1}$, add an edge to $C_i$ between the initial and terminal vertex of any subpath of $C_i$ whose label (respecting the orientation) is $x$. By doing so, for each $i$, we get a new graph which we call $C_i^X$.
The cycles $C_i$ embed in $\Cay(G,S)$, and $\Cay(G,X)$ is precisely the graph formed by replacing each embedded copy of $C_i$ with an embedded copy of $C_i^X$.  

Recall that the {\bf poset of thin cones} $\mathcal{TC}(\mathcal P)$ is the subset of all $X\in\mathcal G^4_L(\mathcal P)$ with the property that there exists a constant $\delta\geq 0$ such that for each $i$, $C_i^X$ is $\delta$-hyperbolic.

We call each embedded copy of $C_i^X$ in $\Cay(G,X)$ a {\bf cone}, the copies of $C_i$ in $\Cay(G,X)$ the {\bf join} of the cone, and the added edges {\bf cone edges}.  To distinguish between embedded copies of $C_i$ in $\Cay(G,S)$ or $\Cay(G,X)$ and $C_i$ as a component of $\Gamma$, we call the embedded copies of $C_i$ {\bf relators}.  Each relator $R$ is the join of a unique cone which we denote $R^X$.  We call  edges of $\Gamma(G,X)$ labeled by elements of $X\setminus S$ {\bf cone edges}; note that these are the images of cone edges in $C^X_i$ under the embedding into $\Cay(G,X)$.  We refer to edges that are not cone edges (that is, the edges of $\Cay(G,S)$) as {\bf $S$--edges}.  We use $d_S$ and $d_X$ to denote the natural metrics on $\Cay(G,S)$ and $\Cay(G,X)$, respectively.

\begin{remark} It is sufficient to replace the requirement that $X\supseteq P^4$ with the following weaker condition: the diameter of any piece in any $C_i$ has uniformly bounded diameter in $C_i^X$. The graph $\Cay(G,X')$, where $X'=X\cup P^4$, is $G$-equivariantly quasi-isometric to $\Cay(G,X)$, and so $X\sim X'$, which implies $[X]\in\mathcal{TC}(\mathcal P)$.
\end{remark}

The obvious example of a thin cone is $[L]$, the coned-off graph considered by Gruber and Sisto \cite{GruberSisto}. Before continuing with more of the theory, let us pause to give one other important example of a thin cone.

\begin{example}\label{def:lacedcone1} Let $\mathcal{P}=\fpres{S}{R}$ be a $C'(\frac{1}{24})$ presentation. Enumerate $R=\set{r_1,r_2,\ldots}$, let $C_i$ be a cycle with label $r_i$, set $\overline{C_i}=C_i^{P^4}$, fix a vertex $x_i$ in each $\overline{C_i}$ and define $\mathcal{P}_i$ to be the set of all paths in $C_i$ which connect two points $y,z$ such that $d_{\overline{C_i}}(x_i,y)=d_{\overline{C_i}}(x_i,z)$. Now set
\[
 X((x_i)_i) = S\cup P^4 \cup \bigcup_{i\geq 1} \setcon{\Lab(P)}{P\in\mathcal P_i} \subseteq L.
\]
It is easy to see that the resulting graphs $C_i^X$ are always $1$--hyperbolic, and therefore $[LC((x_i)_i)]\in\mathcal{TC}(\mathcal P)$. We call $LC((x_i)_i)$ the \textbf{laced cone based at $(x_i)_i$}.
\end{example}

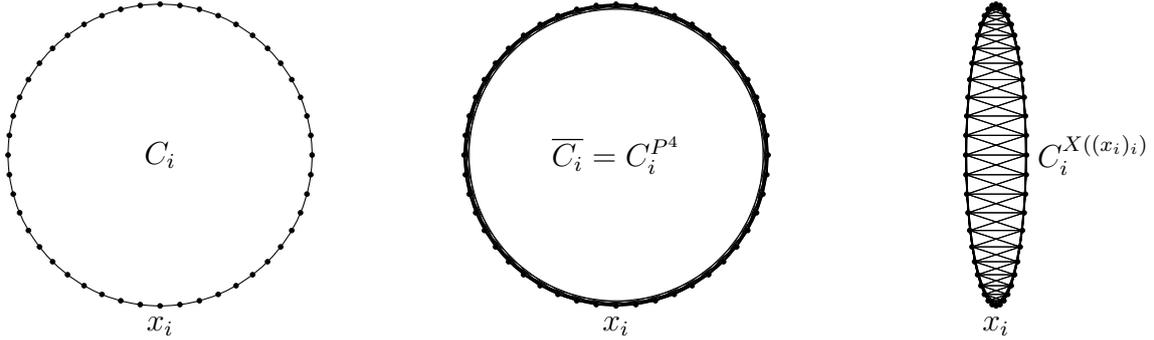
\begin{figure}[H]
\centering
\begin{tikzpicture}[yscale=1,xscale=1, 
vertex/.style={draw,fill,circle,inner sep=0.3mm}]

\draw[thin] (0,0) circle (2cm);

\foreach \r in {0,...,47}{
 \filldraw[]
  (0,0) +(7.5*\r:2cm) circle (0.3mm);
  }

\path (0,0) node[] {$C_i$};
\path (0,-2) node[below] {$x_i$};
\begin{scope}[xshift=6cm]
\draw[thin] (0,0) circle (2cm);

\foreach \r in {0,...,47}{
 \filldraw[]
  (7.5*\r:2cm) circle (0.3mm);
 \draw[very thin]
  (7.5*\r:2cm) -- (7.5*\r+7.5:2cm)
(7.5*\r:2cm) -- (15+7.5*\r:2cm)
(7.5*\r:2cm) -- (22.5+7.5*\r:2cm)
(7.5*\r:2cm) -- (30+7.5*\r:2cm);
  }

\path (0,0) node[] {$\overline{C_i}=C_i^{P^4}$};
\path (0,-2) node[below] {$x_i$};
\end{scope}

\begin{scope}[xshift=11cm, xscale=0.2]
\draw[thin] (0,0) circle (2cm);

\foreach \r in {0,...,47}{
 \filldraw[]
  (7.5*\r:2cm) ellipse (1.5mm and 0.3mm);
 \draw[very thin]
  (7.5*\r:2cm) -- (7.5*\r+7.5:2cm)
(7.5*\r:2cm) -- (15+7.5*\r:2cm)
(7.5*\r:2cm) -- (22.5+7.5*\r:2cm)
(7.5*\r:2cm) -- (30+7.5*\r:2cm);

 \draw[]
 (7.5*\r:2cm) -- (180-7.5*\r:2cm)
 (7.5*\r+7.5:2cm) -- (180-7.5*\r:2cm)
 (7.5*\r-7.5:2cm) -- (180-7.5*\r:2cm);
  }

\path (2,0) node[right] {$C_i^{X((x_i)_i)}$};
\path (0,-2) node[below] {$x_i$};
\end{scope}

\end{tikzpicture}
\caption{The construction of the laced cone based at $(x_i)_i$.}\label{fig:lacedcone}
\end{figure}

\subsection{Associated paths}
In order to prove the hyperbolicity of spaces $\Cay(G,X)$ we will show that all bigons are uniformly thin, and to prove acylindricity of actions $G\curvearrowright \Cay(G,X)$ we will need to study geodesic quadrangles. In preparation for both cases, we will begin by constructing combinatorial geodesic bigons and quadrangles in $\Cay(G,S)$ ``associated to'' geodesic bigons and quadrangles in $\Cay(G,X)$ using the small cancellation assumptions.  In order to do this, we must first define paths in $\Cay(G,S)$ which are associated to geodesics in $\Cay(G,X)$.  In this subsection we present the construction of such paths and record a few of their properties.

Let $\gamma$ be any geodesic in $\Cay(G,X)$. Number the vertices $x_0,\ldots,x_m$ in the order they occur on $\gamma$ (in particular, this means that $d_{X}(x_i,x_j)=\abs{i-j}$).

Choose a subsequence $x_{i_j}$ of the $x_i$ such that $x_{i_0}=x_0$ and, for each $j\geq 1$, $x_{i_j}$ is the last vertex in the sequence such that $x_{i_{j-1}},x_{i_{j-1}+1},\ldots,x_{i_j}$ all lie on some common relator $R_j$. By construction, $i_j>i_{j-1}$, since our standing assumption is that any pair of neighbouring vertices in $\Cay(G,X)$ are contained in some common relator.

For convenience we write $y_j=x_{i_j}$. For each geodesic $\gamma$ in $\Cay(G,X)$, we fix the following data:
\begin{itemize}
\item a choice of relators $R_1,\ldots,R_l$ satisfying the above conditions; and
\item a choice of geodesics $P_j\subset R_j$ from $y_{j-1}$ to $y_j$.
\end{itemize}
We now define a path $P$ in $\Cay(G,S)$ (which \textit{a priori} is not embedded) from $y_0=x_0$ to $y_l=x_m$ to be the concatenation of geodesics $P_j$ from $y_{j-1}$ to $y_j$ in $\Cay(G,S)$.  We call $P$ the {\bf $S$--path associated to $\gamma$}.

While there appear to be many choices here, they will not make much difference. If there is a choice of relators $R_j$ and $R_{j'}$, then $y_{j-1}$ and $y_j$ lie on the piece $R_j\cap R_{j'}$, so $d_{X}(y_{j-1},y_j)=1$. If there is a choice of geodesics $P_j\subset R_j$, then $y_{j-1}$ and $y_j$ lie at antipodal points on $R_j$, and we must choose which way to go around this relator.

\begin{figure}
\def\svgwidth{4in}  
  \centering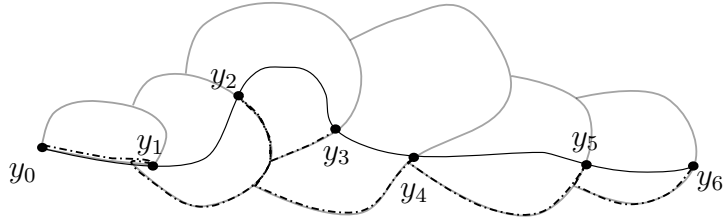 \\
	\caption{The dotted path is an $S$--path associated to $\gamma$ (solid), a geodesic in $\Cay(G,X)$.  Outlined in grey are relators $R_1,\dots, R_6$.} \label{fig:Spath}
\end{figure}

We now consider self-intersections of $P$, which come in two types: 
\begin{enumerate}
\item a single closed subpath $Q$ of $P$ whose image in $\Cay(G,X)$ is a tree;
\item a pair of subpaths $Q_1,Q_2$ of $P$ with the same endpoints, such that the image of $Q_1\cup Q_2$ in in $\Cay(G,X)$ is a tree.
\end{enumerate}

We call a self-intersection of type (i) \textbf{simple} if the initial/terminal vertex of $Q$ is not contained in the interior of $Q$.

The complexity of these self-intersections is limited by the following proposition, which we will prove via a series of lemmas.

\begin{proposition}\label{prop:limitselfint} Any self-intersection of type (i) is contained in a union of at most $4$ consecutive $P_i$. There are no self-intersections of type (ii).
\end{proposition}

We begin with a simple observation about the choice of relators.

\begin{lemma}\label{diffrels} If $R_k=R_l$, then $k= l$ or $\abs{l-k}\geq 4$.
\end{lemma}
\begin{proof}
Suppose $k\neq l$, and assume without loss of generality that $l>k$. It is clear that consecutive relators defining the $S$-path are distinct by construction, so $l-k\neq 1$.  If $l-k=2$, then $y_k$ and $y_{k+1}$ are both in $R_k\cap R_{k+1}$ and so are connected by an edge in $\Cay(G,X)$.  Thus they are consecutive vertices in $\gamma$, which contradicts the choice of $y_k$ in the construction of the $S$-path. Finally, if $l-k=3$ we claim that $R_k\cap R_{k+1}\cap R_{k+2}$ contains a vertex.  To see this, notice that if it does not, then there is a simple geodesic triangle $T$ whose boundary consists of subgeodesics (in $\Cay(G,S)$) of $R_k$, $R_{k+1}$ and $R_{k+2}$, but this is easily seen to be impossible using the small cancellation assumption and Strebel's classification (Theorem \ref{thm:Strebel}).  Specifically, by Lemma \ref{lem:Greendl} there will always be a face $F$ in any reduced diagram $D$ with boundary $T$ such that the boundary of $F$ is contained in: a union of a geodesic and at most 3 pieces, if $\partial F$ is equal to one of $R_k$, $R_{k+1}$, $R_{k+2}$; or at most 5 pieces if not, both of which contradict the small cancellation assumption. It follows that there is a path from $y_k$ to $y_{k+2}$ contained in the union of pieces $R_k\cap R_{k+1}$ and $R_{k+2}\cap R_{k+3}$. Since $R_k=R_{k+3}$, $2\leq d_{X}(y_k,y_{k+2})\leq 1$, which is a contradiction.
\end{proof}

Equipped with this we can now start limiting the self-intersections of type (i).

\begin{lemma}\label{treemeets} Suppose $P$ contains a closed subpath $Q$ of type (i) which intersects each of $P_i,\ldots,P_j$ (not necessarily consecutive indices) in at least an edge, and is contained in their union. Then either $j-i\leq 3$ or there exist $k,l$ with $i\leq k\leq l \leq j$ and $l-k\in\set{4,5,6}$ such that each of $P_k$ and $P_l$ intersect $Q$ in an edge, and $P_k\cap P_l$ contains a vertex in $Q$.
\end{lemma}
\begin{proof} Let $o$ be the initial/terminal vertex of $Q$, which is contained in $P_i$ and $P_j$. Let us suppose for a contradiction that $j-i\geq 4$, and that for all $k,l$ with $i\leq k\leq l \leq j$ and $l-k\in\set{4,5,6}$, $P_k\cap P_l=\emptyset$. We take a closed subpath of $Q$ with $j-i$ minimal; that is, if $i\leq i'\leq j'\leq j$, $j'-i'\geq 4$ and $P_{i'}\cap P_{j'}\neq\emptyset$, then $i=i'$ and $j=j'$.

First, suppose $P_i\cap Q$ is contained in $P_{i+1}\cup P_{i+2}\cup P_{i+3}$. Then either $j-i\leq 6$, in which case the conclusion holds with $k=i$ and $l=j$, or we can find a closed subpath of $Q$ contradicting the minimality assumption.  In the latter case, we consider $o$ as a point on some $P_{i'}$ with $1\leq i'-i\leq 3$ and take the closed subpath of $Q$ starting at $o$ which contains the edge of $P_i\cap Q$ with end vertex $o$ and ending at the original end of $Q$ in $P_j$.

Otherwise, let $p$ be the point closest to $o$ contained in $P_i\cap Q\cap (P_{i+1}\cup P_{i+2}\cup P_{i+3})$.  Consider a closed subpath of $T$ with initial and terminal vertex $p$ which starts in some $P_{i'}$ with $1\leq i'-i\leq 3$ and finishes back at $p$ in some $P_{j'}$ which contains the unique edge in $Q$ with end vertex $p$ on the geodesic connecting $p$ to $o$ in $Q$. It follows by assumption that $j'\geq i+4$. By minimality $j'-i'\leq 3$, so $4\leq j'-i\leq 6$, in which case the result holds with $k=i$ and $l=j'$, since $p\in P_i\cap P_{j'}$ and both $P_i$ and $P_{j'}$ contain an edge in $Q$.
\end{proof}

\begin{lemma}\label{no456trees} Suppose $P$ contains a self-intersection $Q$ of type (i). If $Q\cap P_i$ and $Q\cap P_j$ contain edges, $Q\cap P_i\cap P_j$ contains a vertex, and $\abs{i-j}\leq 6$, then $\abs{i-j}\leq 3$.
\end{lemma}

\begin{proof} Firstly suppose $\abs{i-j}=4$ and consider a path in $Q$ from $y_i$ to $y_{i+3}$.  A case analysis of the different possible configuration of such a subtree (see Figure \ref{fig:treecases}) shows that this path can be decomposed into two paths such that the first is contained in $\bigcup_{l=1}^3 (P_i\cap P_{i+l})$, and the other is contained in $\bigcup_{l=1}^3 (P_j\cap P_{j-l})$.   By Lemma \ref{diffrels}, each of these is a union of at most 3 pieces contained in $R_i$ and $R_j$ respectively. Thus $3\leq d_{X}(y_i,y_{i+3})\leq 2$ which is a contradiction.

\begin{figure}[h!] 
\centering
\subfloat[]{
  \resizebox{2.2in}{!}{
  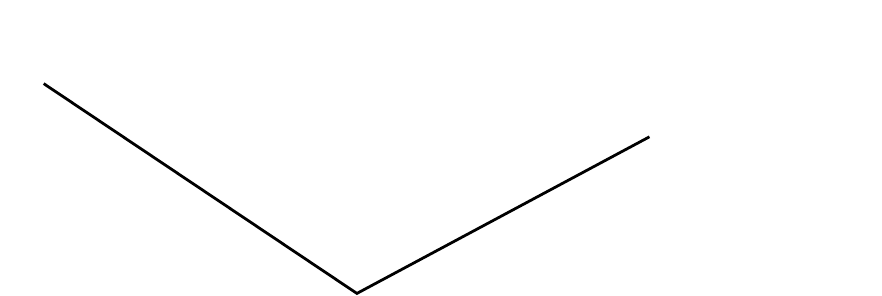}
}
\subfloat[]{
  \resizebox{2.2in}{!}{
  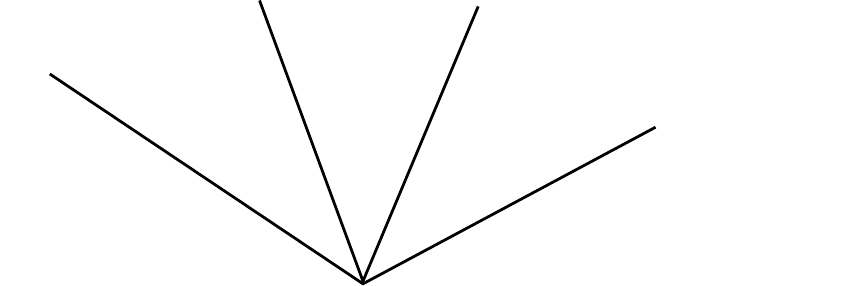}
}
\subfloat[]{
  \resizebox{2.2in}{!}{
  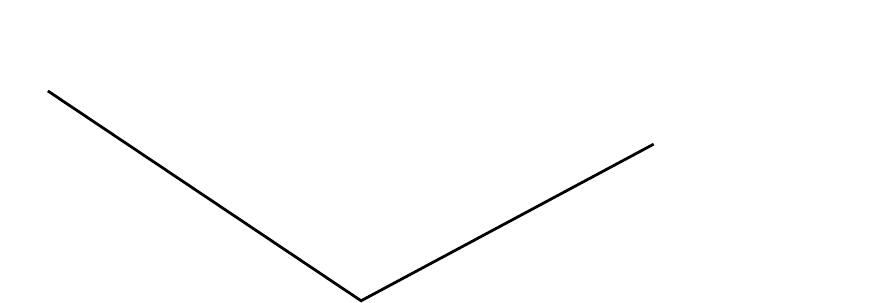}
}
\hspace{0mm}
\subfloat[]{
   \resizebox{2.2in}{!}{
  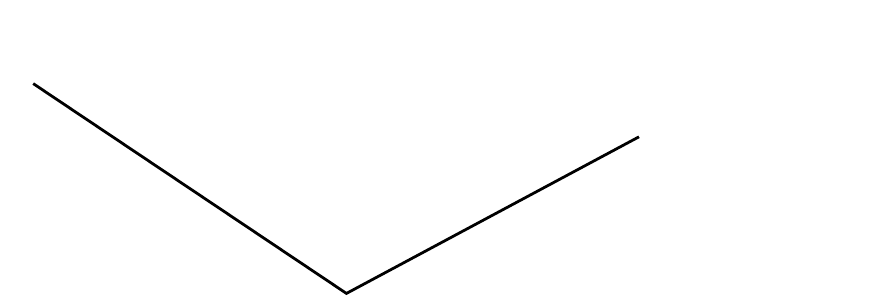}
}
\subfloat[]{
  \resizebox{2.2in}{!}{
  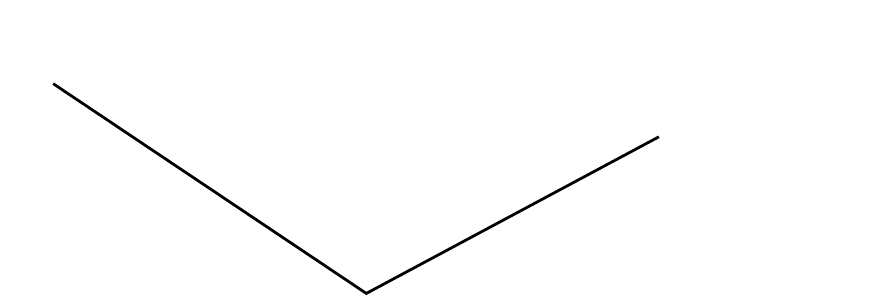}
}

\subfloat[]{
  \resizebox{2.2in}{!}{
  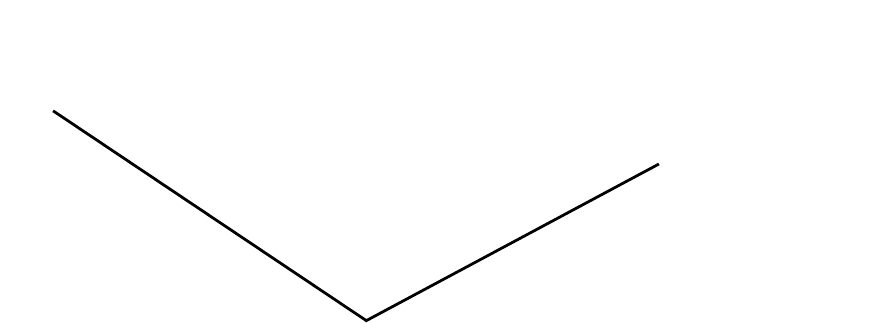}
}
\hspace{0mm}
\subfloat[]{

  \resizebox{2.2in}{!}{
  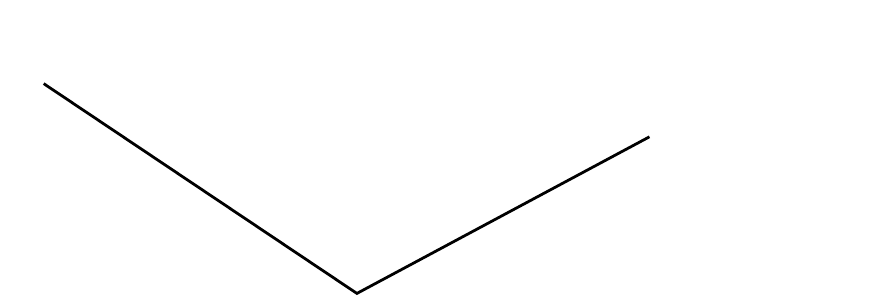}
}
\subfloat[]{  
  \resizebox{2.2in}{!}{
  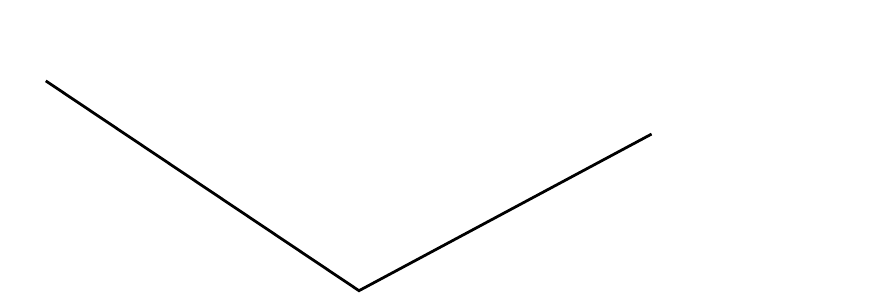}
}
\subfloat[]{
  \resizebox{2.2in}{!}{
  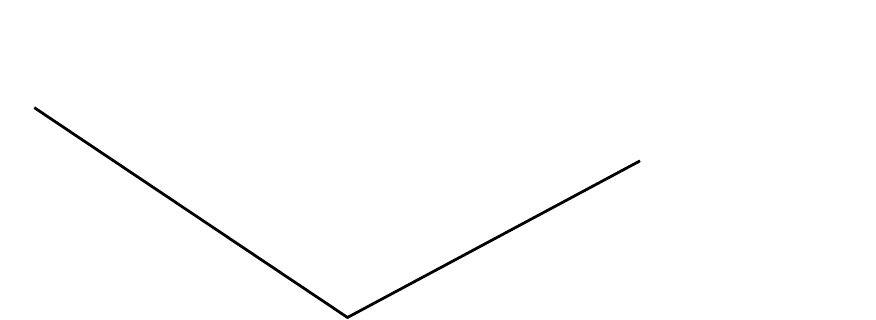}
}
\caption{The 9 possible configurations when $|i-j|=4$.  Note that by Lemma \ref{diffrels}, all $R_k$ in the figure are distinct, except possibly $R_i$ and $R_{i+4}$.} \label{fig:treecases}
\end{figure}

Next suppose $\abs{i-j}=5$.  If $P_1\cap P_{i+4}\neq\emptyset$ or $P_{i+1}\cap P_{i+5}\neq\emptyset$, then by considering the appropriate subtree, we reach a contradiction as above.  Thus we may assume that $P_i\cap P_{i+4}=\emptyset$ and $P_{i+1}\cap P_{i+5}=\emptyset$. Consider a path in $Q$ from $y_i$ to $y_{i+4}$.  As above, a case analysis shows that this path can be decomposed into two paths such that the first is contained in $\bigcup_{l=1}^3 (P_i\cap P_{i+l})$, and the other is contained in $\bigcup_{l=1}^3 (P_j\cap P_{j-l})$.   Thus $4\leq d_{X}(y_i,y_{i+4})\leq 2$ which is a contradiction.

Finally suppose $\abs{i-j}=6$.  If $P_{i'}\cap P_{j'}\neq\emptyset$ and $\abs{i'-j'}\in\set{4,5}$, then we are in one of the previous two cases and reach a contradiction.  Thus we may assume that if $P_{i'}\cap P_{j'}\neq\emptyset$, then $\abs{i'-j'}\in\set{0,1,2,3,6}$.  Consider a path in $T$ from $y_i$ to $y_{i+5}$. A case analysis again shows that this path can be decomposed into two paths such that the first is contained in $\bigcup_{l=1}^3 (P_i\cap P_{i+l})$, and the other is contained in $\bigcup_{l=1}^3 (P_j\cap P_{j-l})$. Thus $5\leq d_{X}(y_i,y_{i+5})\leq 2$ which is a contradiction.
\end{proof}

Combining Lemma \ref{no456trees} with Lemma \ref{treemeets}, we deduce that any self-intersection of type (i) is contained in a union of at most $4$ consecutive $P_i$ as follows.  If $T$ intersects $R_k$ and $R_l$ in edges, and $l-k\geq 4$, then by Lemma \ref{treemeets} there exist $k\leq i < j \leq l$ satisfying the hypotheses of Lemma \ref{no456trees}, and so $l-k\leq 3$, which is a contradiction.

In order to deal with self-intersections of type (ii), we introduce a refinement of the $S$--path associated to $\gamma$.  The {\bf essential $S$--path associated to $\gamma$}, which we denote by $P_{ess}$, is formed from $P$ by removing the interiors of all self-intersections of type (i).  By definition $P_{ess}$ is connected and has the same end vertices as $P$ (and therefore as $\gamma$, as well). We now prove that $P_{ess}$ is embedded in $\Cay(G,S)$.

\begin{figure}
\def\svgwidth{4in}  
  \centering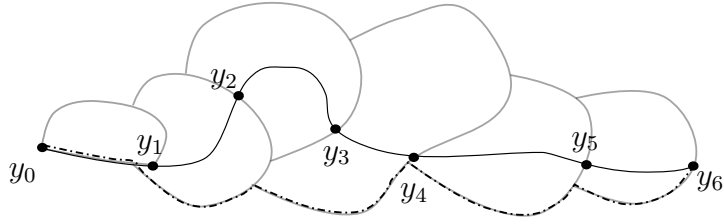 \\
	\caption{The dotted path is an essential $S$--path associated to a $\gamma$ (solid).  Compare this to the $S$--path in Figure \ref{fig:Spath}.}
\end{figure}

\begin{lemma} \label{disttoPess}
Any point on a self-intersection of type (i) is at distance at most $2$ from $P_{ess}$ in $\Cay(G,X)$.  Moreover, if $P_i\cap P_{ess}\neq\emptyset$, then every vertex on $P_i\setminus P_{ess}$ is connected to $P_{i}\cap P_{ess}$ by an edge in $R^X_i$. 
\end{lemma}

\begin{proof}
That the distance from any point on a self-intersections of type (i) to $P_{ess}$ is uniformly bounded follows immediately from the fact that any such self-intersection is contained in a union of at most 4 consecutive $P_i$ and that the initial/terminal vertex of such a closed subpath must, by definition, lie on $P_{ess}$.  That the bound is 2 follows from a case analysis of the possible configurations of such trees.  We illustrate one instance of the worst-case scenario  in Figure \ref{fig:worstcase} below. To prove the last statement, notice that if $P_i\cap P_{ess}\neq\emptyset$ then there is a path from any vertex in $P_i\setminus P_{ess}$ to $P_i\cap P_{ess}$ contained in $\bigcup_{l=1}^3 R_i\cap R_{i+l}$ or in $\bigcup_{l=1}^3 R_i\cap R_{i-l}$, and so by Lemma \ref{diffrels} they are connected by an edge in $R_i^X$.
\end{proof}

\begin{figure}[h!]
 
  \resizebox{3in}{!}{
  \centering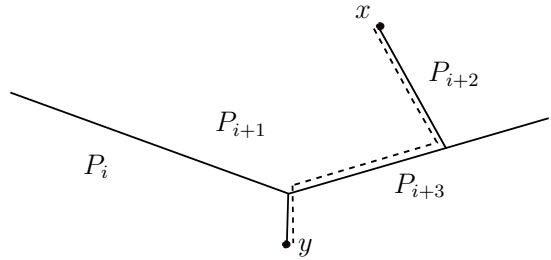} \\
    	\caption{There is a path (dotted) from $x$ to a point $y\in P_{ess}$ consisting of a piece $R_{i}\cap R_{i+3}$ and the union of two pieces  $R_{i+1}\cap R_{i+2}$ and $R_{i+1}\cap R_{i+3}$, so the distance in $\Cay(G,X)$ from $x$ to $y$ is 2.}\qedhere \label{fig:worstcase}
	\end{figure}

\begin{lemma} Let $P$ be an $S$-path constructed from geodesics $P_i$ in relators $R_i$ and let $P_{ess}$ be the corresponding essential $S$-path. If there exists an $i$ such that $P_i\cap P_{ess}=\emptyset$, then $d_{X}(y_{i-1},y_i)= 1$.  Moreover, at most two consecutive $P_i$'s can be disjoint from $P_{ess}$.
\end{lemma}
\begin{proof} If $P_i\cap P_{ess}=\emptyset$ then $P_i$ is contained in a self-intersection of $P$ of type (i) with the additional property that the initial/terminal vertex of this self-intersection  which is itself contained in a union of paths $P_j,\ldots,P_{j+l}$ with $l\leq 3$ by Lemmas \ref{treemeets} and \ref{no456trees}. It is clear that $P_j\cap P_{ess}$ and $P_{j+l}\cap P_{ess}$ are not empty, so $i\neq j,j+l$ .  Thus we have at most 2 choices for $i$, and, it follows that  $P_i$ is contained in a union of at most 3 pieces: the intersections with the other paths in this self-intersection of $P$. Hence $d_{X}(y_{i-1},y_i)=1$.
\end{proof}

Before continuing, we pause for one remark. The $1$-skeleton of a reduced diagram can be naturally embedded as a labelled subgraph of $\Cay(G,S)$, so given a face $F$ in a reduced diagram $D$ whose boundary is explicitly fixed as a closed path in $\Cay(G,S)$, we may make sense of the statement $\partial F$ is equal (or not equal) to a certain relator $R$. If $\partial F=R$ and $R$ is the join of a unique cone $R^X$, we may write $(\partial F)^X$ for this cone.  This will be invaluable when constructing diagrams to bound $S$--distances using the assumptions on thin cones. 

\begin{lemma} \label{numberoffacesbound}
Let $Q$ be any closed path in $\Cay(G,S)$ and let $D$ be any reduced diagram with boundary $Q$. Let $F$ be a face of $D$ with $e(F)\geq 1$.  Suppose there is a subpath $\alpha$ of $\partial D\cap \partial F$ such that $\alpha\subseteq P_{ess}$, and let $I$ be the set of indices $i$ such that $P_i\cap \alpha\neq \emptyset$, with minimal and maximal elements $i_1,i_2$, respectively.
	\begin{enumerate}
	\item If $\partial F\neq R_i$ for all $i\in I$, then $i_2-i_1\leq 4$.
	\item If $\partial F=R_i$ for some $i\in I$, then $\partial F\neq R_j$ for all $j\in I\setminus\{i\}$.
	\item If, in addition, $\partial F\setminus \alpha$ is contained in a union of at most $N$ pieces, then $\abs{I}\leq 4+\lceil \frac{N}{4}\rceil$.
	\end{enumerate}
\end{lemma}
\begin{proof}
Assume $i_2-i_1\geq 2$. By Lemma \ref{disttoPess}, the vertices $y_{i_1}$ and $y_{i_2-1}$ are connected by an edge in $\Cay(G,X)$ to points $z_1\in \alpha\cap R_{i_1}$ and $z_2\in\alpha\cap R_{i_2-1}$, respectively.  There is a path in $\partial F$ from $z_1$ to $z_2$ which is contained in the union of pieces $\partial F\cap R_k$ for $k\in I$ and $i_1<k\leq i_2-1$.  Therefore, we have 
\begin{equation}
(i_2-1)-i_1\leq d_{X}(y_{i_1},y_{i_2-1})\leq 2+\left\lceil \frac{(i_2-1)-i_1}{4}\right\rceil,
\end{equation}
and so $i_2-i_1-1\leq 3$ and part (i) follows.

To prove (ii), suppose that $\partial F = R_i=R_j$ for some $i<j$ and $R_k\neq F$ for all $k\in I$ satisfying $i<k<j$.  In this case, $y_i$ and $y_{j-1}$ lie on $\partial F$.  By a slight modification of the argument in (i), we see that
\begin{equation}
(j-1)-i_1\leq d_{X}(y_{i},y_{j-1})\leq \left\lceil \frac{(j-1)-i_1}{4}\right\rceil,
\end{equation} so $j-i\leq 2$. This contradicts Lemma \ref{diffrels}.

For part (iii), first suppose $\partial F\not\in\set{R_{i_1},R_{i_2}}$. By Lemma \ref{disttoPess} there are paths which are a union of at most 3 pieces in $R_{i_1}$ (resp. $R_{i_2}$) connecting $y_{i_1}$ and $y_{i_2-1}$ to two vertices of $\alpha$, extending these by pieces in $\partial F\cap R_{i_1}$ and $\partial F\cap R_{i_2}$ respectively, we see that $y_{i_1}$ and $y_{i_2}$ can be connected by edges in $\Cay(G,X)$ to the end vertices of $\alpha$. Therefore
\[
 (i_2-1)-i_1 \leq d_{X}(y_{i_1},y_{i_2-1}) \leq 2 + \left\lceil \frac{N}{4}\right\rceil,
\]
so $i_2-i_1\leq 3+\lceil \frac{N}{4}\rceil$ and $\abs{I}\leq 4+\lceil \frac{N}{4}\rceil$.

Next suppose $\partial F = R_{i_1}$; the case $\partial F = R_{i_2}$ is handled similarly.  By (ii), $\partial F\neq R_j$ for all $j\in I\setminus\set{i_1}$. There is an edge connecting $y_{i_2-1}$ to $\partial F$ so
\[
(i_2-1) - i_1 \leq d_{X}(y_{i_1},y_{i_2-1}) \leq 1 + \left\lceil \frac{(i_2-1) - i_1}{4}\right\rceil,
\]
so $i_2-i_1\leq 3$ and $\abs{I}\leq 4$.
\end{proof}

We are now ready to prove Proposition \ref{prop:limitselfint}.

\begin{proof}[Proof of Proposition $\ref{prop:limitselfint}$]
The first statement  follows immediately from Lemmas \ref{treemeets} and \ref{no456trees}.  

If $P$ has a self-intersection of type (ii), then $P_{ess}$ must contain a simple cycle.  We will show that this is not possible.  Suppose for a contradiction that $C$ is a simple cycle in $P_{ess}$ and let $D$ be a reduced diagram with boundary $C$. 

Since $D$ contains a face which contributes positively to the curvature formula (\ref{Strebcurv}), either there is a face $F$ in $D$ with $e(F)=1$ and $1\leq i(F)\leq 3$, or $D$ is a single face $F$. In either case, applying Lemma \ref{numberoffacesbound}(iii) to $\alpha=\partial F\cap\partial D$, we see that $\alpha$ is contained in a union of either 5 pieces (if $\partial F\neq R_i$ for all $i$), or a geodesic and at most 4 pieces (if $\partial F= R_i$ for some $i$). Therefore, $\partial F$ is contained in a union of either 8 pieces, or a geodesic and at most 7 pieces.  Since the presentation $\mathcal P$ is $C'(\frac{1}{24})$, we see that
\[
 \abs{\partial F} < \abs{\partial F} \max\set{\frac12 + \frac{7}{24},\frac{8}{24}}
\]
which is a contradiction.
\end{proof}

In order to be able to utilise the small cancellation assumptions we will need to construct combinatorial geodesic polygons.  The following proposition shows that our essential $S$-paths are suitable for this purpose.

\begin{proposition}\label{prop:combgeods} Let $Q$ be a closed path in $\Cay(G,S)$ which is a union of paths $P^1,\ldots,P^n$ each of which is either a geodesic in $\Cay(G,S)$ or an essential $S$-path of a geodesic in $\Cay(G,X)$. Let $D$ be a reduced diagram with boundary $Q$. Then $D$ is a (possibly degenerate) combinatorial geodesic $n$-gon (cf. Definition $\ref{def:combinatorialgeodesicpolygon}$).
\end{proposition}

\begin{proof}
It suffices to show that any face $F\subseteq D$ whose exterior side is completely contained in some $P^t_{ess}$ satisfies $i(F)\geq 4$.  Towards a contradiction, suppose $i(F)\leq 3$.  

Then by Lemma \ref{numberoffacesbound}, $\partial F\cap \partial D$ is contained in a union of either 5 pieces (if $\partial F\neq F^t_i$ for any $i$) or a geodesic and at most 4 pieces (if $\partial F=R^t_i$ for some $i$).  Thus $\partial F$ is contained in a union of either 8 pieces or a geodesic and at most 7 pieces.  Thus we obtain a contradiction \[\abs{\partial F} <\abs{\partial F} \max\set{\frac12+\frac{7}{24},\frac{8}{24}}.\]
\end{proof}

A very useful consequence of this is that cones are convex in $\Cay(G,X)$.

\begin{lemma} \label{lem:qcxity}
Let $R$ be a relator in $\Cay(G,X)$. Any  geodesic in $\Cay(G,X)$ with endpoints in $R^X$ is contained in $R^X$.
\end{lemma}
\begin{proof}
Suppose for a contradiction that there is a geodesic $\gamma'$ of positive length in $\Cay(G,X)$ such that both end vertices $x,x'$ of $\gamma'$ are contained in a cone $R^X$ and $\gamma'\cap R^X=\set{x,x'}$. Consider an essential $S$--path $P'$ associated to $\gamma'$ which is a concatenation of geodesics $P'_1,\ldots P'_l$ contained in relators $R'_1,\ldots,R'_l$. 

Suppose first that $R=R'_i$ for some $i$. Since $\gamma'\cap R^X=\set{x,x'}$, and, by definition, each $R'_i$ contains at least two vertices $y'_{i-1}$ and $y'_i$ of $\gamma'$, we see that $\set{x,x'}=\set{y_{i-1},y_i}$ and therefore all of the vertices in $\gamma'$ are contained in $R^X$. The choice of the relator $R'_i$ then ensures that there are only two vertices in $\gamma'$, so $\gamma'$ is an edge $e=xx'$. Thus $\Lab(x,x')\in X$, and $e\in R^X$, which is a contradiction.

Hence, we may assume $R\neq R'_k$ for all $k$. Let $\gamma_S$ be a geodesic in $\Cay(G,S)$ from $x$ to $x'$ which is contained in $R$, and consider a reduced diagram $D$ with boundary $\gamma_S\cup P'$. If $\gamma_S\subseteq P'$ then $l\leq d_X(x,x')\leq \lceil\frac{l}{4}\rceil$, so $l=1$. Therefore there is a path connecting $x$ to $x'$ in $\Cay(G,S)$ which is a piece in $R\cap R'_1$. Thus $xx'$ is an edge $e$ in $\Cay(G,X)$, $\Lab(x,x')\in X$, and $e\subset R^X$ which is a contradiction. Otherwise, by Proposition \ref{prop:combgeods}, $D$ is a combinatorial geodesic bigon which contains a face, so contains a sub-bigon $D'$ with $\partial D'\subseteq \partial D$ which is either a single face or has the form $I_1$ from Strebel's classification.

Suppose $F$ is a face of $D'$ with $i(F)\leq 1$. If $\partial F\neq R_j'$ for any $j=1,2,\dots, l$, then by Lemma \ref{numberoffacesbound}(i)  $\partial F$ is contained in either a union of at most 7 pieces, or a geodesic and at most 6 pieces. Both of these contradiction the small cancellation assumption. If $\partial F= R_j'$ for some $j$, then by applying Lemma \ref{numberoffacesbound}(iii) to $\alpha=\partial F\cap P'$ we see that $\partial F$ is contained in a union of a geodesic and at most 6 pieces, which is again a contradiction since $\mathcal P$ is a $C'(\frac{1}{24})$ presentation.
\end{proof}
 
We now collect several lemmas about essential $S$--paths that will be useful in the construction of associated bigons and quadrangles, as well as in the proof of hyperbolicity.  

\begin{lemma} \label{lem:essRconn} Let $P_{ess}$ be an essential $S$--path associated to a geodesic $\gamma$ in $\Cay(G,X)$. 
	\begin{enumerate}
	\item  For every relator $R$, $R\cap P_{ess}$ is empty or connected.
	\item  If $R$ is a relator such that $P_{ess}\cap R$ has diameter at least $3$ in $\Cay(G,X)$, then $R=R_j$ for some $R_j$ used in the construction of $P_{ess}$.
	\end{enumerate}
\end{lemma}

\begin{proof} 
To prove (i), suppose there is a relator $R$ such that $R\cap P_{ess}$ is not connected (so it is clear from Lemma \ref{lem:qcxity} that $R$ is not one of the relators $R_i$), and let $Q$ be a positive length subpath of $P_{ess}$ with end vertices $x,x'$ such that $R\cap Q=\set{x,x'}$. Let $D$ be a reduced diagram whose boundary is a simple cycle comprising a geodesic $\gamma$ in $\Cay(G,S)$ which is contained in $R$ and a subpath of $P_{ess}$. Since $\partial D$ is the union of a geodesic in $\Cay(G,S)$ and a subpath of $P_{ess}$, by Proposition \ref{prop:combgeods}, $D$ is a combinatorial geodesic bigon.  By Strebel's classification, $D$ is therefore of type $I_1$.  Choose $F\subseteq D$. Then $\partial F$ can be written as the union of at most 2 interior pieces, $\partial F\cap \gamma$ and $\partial F\cap P_{ess}$. By Lemma \ref{numberoffacesbound}(iii), $\partial F\cap P_{ess}$ is contained in a union of at most 5 pieces (since $R\neq R_i$ for all $i$), and so
\begin{equation}\label{eq:Fbound}
  \abs{\partial F} < \abs{\partial F\cap \gamma} + \frac{7}{24}\abs{\partial F}.
\end{equation}
However, since $\gamma$ is a geodesic in $\Cay(G,S)$, we have $\abs{\partial F\cap \gamma}\leq\frac12 \abs{\partial F}$ which contradicts ($\ref{eq:Fbound}$).

For (ii), suppose $R\neq R_j$ for all $j$. Then by Lemma \ref{numberoffacesbound}(i), $P_{ess}\cap R$ is contained in a union of at most 5 pieces $R\cap R_j$, so the distance between the endpoints of $P_{ess}\cap R$ is at most 2.
\end{proof}

\begin{lemma}\label{lem:Pesslocgeod} Let $P_{ess}$ be an essential $S$--path associated to a geodesic $\gamma$ in $\Cay(G,X)$. For every relator $R_i$ used in the construction of the path we have $\abs{P_{ess}\cap R_i}< \frac34\abs{R_i}$.  Moreover, if $R_i\cap P_{ess}\neq \emptyset$, then $y_i$ and $y_{i-1}$ can be connected by an edge in $\Cay(G,X)$ to the initial and terminal vertices of $P_{ess}\cap R_i$, respectively.
\end{lemma}
\begin{proof} Consider all the vertices $x_j,\ldots,x_k$ contained in $R_i$. Since the set of indices appearing in this list is a subinterval of $\set{0,\ldots,m}$ by Lemma \ref{lem:qcxity}, we have $y_i=x_k$. Moreover, if we define the set of all vertices contained in $R_{i-1}$ to be $x_p,\ldots,x_q$, we have $q\leq j+1$, since for any two distinct relators, $R$ and $S$ if $x_k,x_l\in R\cap S$ with $k<l$ then there is a path in $R\cap S$ connecting $x_k$ to $x_l$ and so $l-k \leq d_{X}(x_k,x_l)= 1$.  

Moreover, $q<k$, for  otherwise we contradict the choice of $R_i$ in the construction of the $S$--path associated to $\gamma$. Hence, $y_{i-1}\in R_{i-1}\cap R_i\cap \set{x_j,x_{j+1}}$. 

Let $z,z'$ be the initial and terminal endpoints of $R_i\cap P_{ess}$, respectively.  If $y_{i-1}$ does not lie on $P_{ess}$, then it follows from the proof of Lemma \ref{disttoPess} that $y_{i-1}$ is connected by an edge to the point $z$.  Similarly, if $y_i$ does not lie on $P_{ess}$, then it is connected to $z'$ by an edge.  

If $y_{i-1}$ lies on $P_{ess}$, we will show that there is a path connecting $z$ to $y_{i-1}$ which is contained in the union of at most 3 pieces $R_{i'}\cap R_i$, and thus $y_{i-1}$ can be connected to $z$ by an edge.  To this end, assume that any geodesic from $z$ and $y_i$ cannot be connected by the union of at most 3 pieces $R_{i'}\cap R_i$, and let $j$ be the smallest index such that $y_j$ is connected to $P_{ess}\cap R_i$ by an edge.  By assumption, $j\leq i-4$.  However, this implies that $3\leq d_X(y_{j},y_{i-1})\leq 1+\lceil\frac34\rceil$, which is a contradiction.  Similarly, if $y_i$ lies on $P_{ess}$, then there is a path connecting $z'$ to $y_i$ which is contained in the union of at most 3 pieces $R_i\cap R_{i'}$, and so $y_i$ can be connected to $z'$ by an edge.

It follows that $P_{ess}\cap R_i$ is contained in the union of the geodesic $P_i$ and at most six pieces $R_i\cap R_{i'}$.  Since $P_{ess}\cap R_i$ is connected by Lemma \ref{lem:essRconn}(i), we have $\abs{P_{ess}\cap R_i}< (\frac12 + \frac{6}{24})\abs{R_i}$.

\end{proof}

\section{Hyperbolicity of coned-off graphs}\label{sec:hyperbolicity}
The main goal of this section is the following theorem.

\begin{theorem}\label{thm:thinconehyp} Let $\mathcal P=\fpres{S}{r_1,r_2,\dots}$ be a $C'(\frac{1}{24})$ presentation of a group $G$.  Then $\TCG\subseteq \mathcal H(G)$.
\end{theorem}

To prove the theorem, we show for any thin cone $X$, every geodesic bigon in $\Cay(G,X)$ satisfies (\ref{graphhyp}) with $\delta'=7+2\delta$, where $\delta$ is a hyperbolicity constant of the cones $C_i^X$ in the sense of (\ref{spacehyp}): for every geodesic triangle in $C_i^X$ each side is contained in the closed $\delta$--neighborhood of the other two. This suffices to deduce hyperbolicity by \cite[Theorem 2]{Papas_bigons}.

\subsection{Essential $S$--bigons}

Given two vertices $x,x'\in\Cay(G,X)$ which are contained in distinct cones, and a pair of geodesics  $\gamma^1$ from $x$ to $x'$ and $\gamma^2$ from $x'$ to $x$ in $\Cay(G,X)$.  We define an {\bf $S$--bigon corresponding to $(\gamma^1,\gamma^2)$} to be $(P^1, P^2)$ where $P^1$ and $P^2$ be $S$--paths corresponding to $\gamma^1$ and $\gamma^2$, respectively.  We analogously define {\bf essential $S$--bigons $(P^1_{ess}, P^2_{ess})$ corresponding to $(\gamma^1, \gamma^2)$}. Note that $P^t$, $P^t_{ess}$ and $\gamma^t$ have the same endpoints, and thus $(P^1, P^2)$ and $(P^1_{ess}, P^2_{ess})$ are bigons in $\Cay(G,X)$. Moreover, by Proposition \ref{prop:combgeods}, $(P^1_{ess}, P^2_{ess})$ is a combinatorial geodesic bigon in $\Cay(G,S)$.

We append the superscript $t=1,2$ to any notation already defined in Section \ref{sec:conedoffgraph} for an (essential) $S$--path; for example, the vertices of $\gamma^t$ will be denoted $x_i^t$ and relators used in the construction of $P^t$ will be denoted $R_i^t$.

The essential $S$--bigon $P^1_{ess}\cup P^2_{ess}$ is composed of (possibly degenerate) maximal subpaths $S_1,\dots, S_k$ contained in $P^1_{ess}\cap P^2_{ess}$ and simple cycles $B_1,\dots, B_l$, which are formed by taking closures of connected components of $P^1_{ess}\cup P^2_{ess}\setminus \bigcup_{i=1}^k S_i$. Let $\mathcal S=\{S_1,\dots, S_k\}$ and $\mathcal B=\{B_1,\dots, B_l\}$.

\begin{figure}[H]
\def\svgwidth{4in}  
  \centering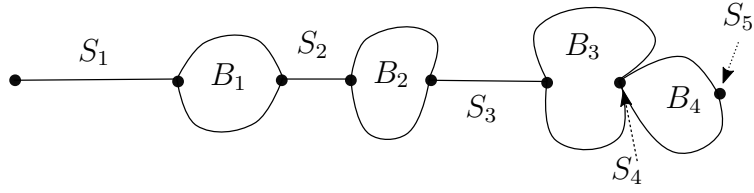 \\
  	\vspace*{5mm}
	\caption{Subdividing the bigon $A_S$ into pieces.  The (possibly degenerate) segments are elements of $\mathcal S$ and the simple cycles are elements of $\mathcal B$.}
\end{figure}

By construction, the simple cycles $B_i$ are bigons.  Moreover, any reduced diagram $D$ with boundary $B_i$ is a combinatorial geodesic bigon (see Definition \ref{def:combinatorialgeodesicpolygon}) by Proposition \ref{prop:combgeods}.

Our goal is to prove that for each $i$, $\gamma^1_i$ is contained in the closed $7+2\delta$--neighborhood of $\gamma^2$.

\begin{lemma}\label{hyppf:selfintersections} Suppose consecutive edges of $P^1_{ess}$ are contained in $R^1_i$ and $R^1_{i'}$ with $i<i'$. Then $d_X(y^1_i,y^1_{i'-1})\leq 2$.
\end{lemma}
\begin{proof} Let $x$ be the common end vertex of the two edges. By Lemma \ref{lem:Pesslocgeod}, $d_X(y^1_i,x)\leq 1$ and $d_X(x,y^1_{i'-1})\leq 1$. The result follows by the triangle inequality.
\end{proof}

By Lemmma \ref{hyppf:selfintersections}, we need only consider $i$ such that $P_{ess}^1\cap R^1_i$ contains an edge, up to increasing the hyperbolicity constant by $1$. We now split into two cases depending on whether the following condition is satisfied.

\textbf{Condition $(\ast)$:} There exists some $j$ such that $R^1_i=R^2_j$ and either
\begin{enumerate}
 \item[$(\ast_a)$] $P^1_{ess}\cap R^1_i \cap P^2_{ess}$ contains an edge; or
 \item[$(\ast_b)$] there is a face $F\subseteq D$ such that each of $\partial F\cap P^1_{ess}\cap R^1_i$ and $\partial F\cap P^2_{ess}\cap R^2_j$  contains an edge.
\end{enumerate}
Notice that $(\ast_a)$ and $(\ast_b)$ are mutually exclusive: if both occurred then $R^1_i$ would contain an edge in a segment and edges on each of $\partial F\cap P^1_{ess}\cap R^1_i$ and $\partial F\cap P^2_{ess}\cap R^2_j$. This is not possible as such a subgraph must have a vertex of degree at least $3$ and $R^1_i$ is a simple cycle.

\begin{figure}[H]
\begin{center}
\resizebox{2.5in}{!}{
 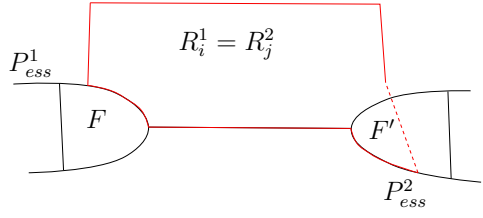} \\
	\caption{One possible configuration satisfying case (a) of $(\ast)$.} \label{fig:daggera}
	\end{center}
\end{figure}

\begin{figure}[H]
\begin{center}
\resizebox{4in}{!}{
 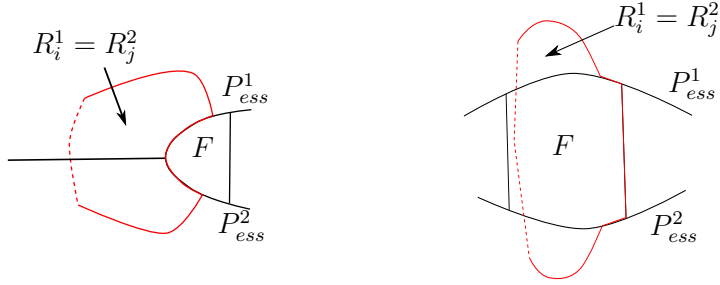} \\
	\caption{Two possible configurations satisfying case (b) of $(\ast)$.} \label{fig:daggerb}
	\end{center}
\end{figure}

\begin{proposition}\label{hyppf:inotj} Suppose $i$ does not satisfy $(\ast)$. Then $\gamma^1_i$ is contained in the closed $6$ neighborhood of $\gamma^2$.
\end{proposition}
\begin{proof}
Let $D$ be a diagram whose boundary is the bigon $A_S$. If $R^1_i=\partial F$ for some face $F\subseteq D$, then, since $(\ast_b)$ fails, $P^2_{ess}\cap \partial F$ is contained in a union of $5$ pieces $\partial F\cap R^2_{j_m}$ with $1\leq m\leq 5$ by Lemma \ref{numberoffacesbound}(i) (see Figure \ref{fig:4.3pf} (left)).  Moreover, $P^2_{ess}\cap\partial F$ cannot be contained in union of at most $4$ pieces in $\partial F$.  To see this, notice that $\partial F$ is the union of $\partial F\cap P^1_{ess}$, which has length less than $\frac34\abs{\partial F}$ by Lemma \ref{lem:Pesslocgeod}, at most 2 internal arcs and $N$ pieces $\partial F\cap R^2_j$, and so
\[
 \abs{\partial F} < \left(\frac34+\frac{2}{24} + \frac{N}{24}\right)\abs{\partial F}.
\] This is a contradiction if $N\leq 4$.
Assume $j_1<\ldots <j_5$. By Lemma \ref{disttoPess} there is an edge in $X$ connecting $y^2_{j_3}$ to some vertex $v$ on $R^2_{j_3}\cap\partial F$, and edges connecting $y^1_{i-1}$ and $y^1_i$ to the end vertices of $\partial F\cap P^1_{ess}$. Each of these end vertices is connected to $v$ by an edge in $X$, since there are paths in $D$ connecting them which are the union of at most $4$ pieces (one internal arc and at most $3$ pieces $R^2_{j_m}\cap\partial F$). Hence
\[
 \max\set{d_X(y^1_{i-1},y^2_{j_3}),d_X(y^1_{i},y^2_{j_3})}\leq 3,
\]
and so $\gamma^1_i \subset B_6(y^2_{j_3})\subset N_6(\gamma^2)$.

\begin{figure}
\begin{center}
\resizebox{5in}{!}{
 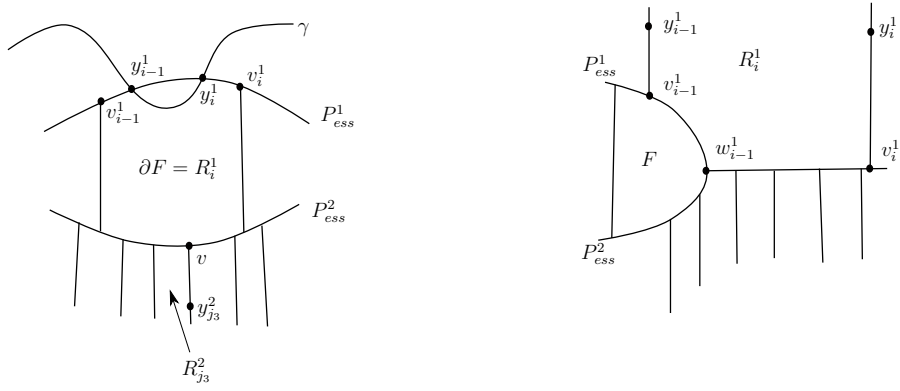} \\
	\caption{Two cases of the proof of Proposition \ref{hyppf:inotj}, when $R^1_i=\partial F$ (left) and when $R^1_i\neq\partial F$ (right).} \label{fig:4.3pf}
	\end{center}
\end{figure}

If $R^1_i\neq\partial F$ for all $F\subset D$, then by splitting $P^1_{ess}\cap R^1_i$ into its intersections with bigons and pieces, then  since $(\ast_a)$ fails, we deduce that $P^1_{ess}\cap R^1_i$ is contained in: some number $M$ of pieces which are the intersection of $R^1_i$ with boundaries of faces $F\subset D$; and some number $N$ of pieces which are intersections of $R^1_i$ and relators $R^2_j$ (see Figure \ref{fig:4.3pf} (right)).  We first show that $M\leq 2$, i.e., there are at most two faces $F_1,F_2$ in $D$ such that $\partial F_m\cap R^1_i$ contains an edge.  Notice that if there were at least three, then since $P^1_{ess}\cap R^1_i$ is connected, there is a face $F\subset D$ such that $\partial F\cap P^1_{ess}$ is contained in the piece $R^1_i\cap \partial F$, and so
\[
\abs{\partial F} < \left(\frac34+\frac{3}{24}\right)\abs{\partial F},
\]
which is a contradiction. We next show that $N\leq 5$.  To see this, notice that if $R^2_{j_1}\cap R^1_i$ and $R^2_{j_2}\cap R^1_i$ contain edges with $j_2>j_1$, then
\[
 (j_2-1)-j_1\leq d_X(y^2_{j_1},y^2_{j_2-1}) \leq 2+ \left\lceil\frac{(j_2-1)-j_1}{4}\right\rceil
\]
so $(j_2-1)-j_1\leq 3$.

We have shown that $P^1_{ess}\cap R^1_i$ is contained in at most two pieces $\partial F_m\cap R^1_i$ and at most $5$ pieces $R^2_j\cap R^1_i$, and so $d_X(y^1_{i-1},y^1_i)\leq 2+\lceil\frac74\rceil=4$.

Define $v^1_{i-1}$ to be an end vertex of $R^1_i\cap P^1_{ess}$ which is connected to $y^1_{i-1}$ by an edge. If $v^1_{i-1}\in P^2_{ess}$, then $v^1_{i-1}$ is connected to an end vertex of some $R^2_j\cap P^2_{ess}$ by a piece.  Thus $d_X(y^1_{i-1},\gamma^2)\leq 3$.

If $v^1_{i-1}\not\in P^2_{ess}$, then it lies in the boundary of a face $F\subseteq D$, and there is a path which is the union of at most 2 pieces in $\partial F$ from $v^1_{i-1}$ to an end vertex $w^1_{i-1}$ of $P^2_{ess}\cap \partial F$.  If $\partial F=R^2_j$ for some $j$, then $d_X(w^1_{i-1},\gamma^2)\leq 1$ by Lemma \ref{lem:Pesslocgeod}.  Otherwise, $\partial F\cap P^2_{ess}$ is a union of exactly five pieces $\partial F\cap R^2_j$ by Lemma \ref{numberoffacesbound}(i), as above.  In this case, there is a path from $w^1_{i-1}$ to an end vertex of some $P^2_{ess}\cap R^2_j$ which is contained in a piece, and so $d_X(y^1_{i-1},\gamma^2)\leq 1+\lceil\frac34\rceil+1=3$.  The same analysis proves that $d_X(y^1_i,\gamma^2)\leq 4$, and thus $\gamma^1_i\subseteq N_5(\gamma^2)$.

\end{proof}

\begin{proposition}\label{hyppf:iandj} Suppose $i$ satisfies $(\ast)$. Then $\gamma^1_i\subseteq N_{6+2\delta}(\gamma^2_j)$.
\end{proposition}
\begin{proof}
It suffices to show that $d_X(y^1_{i-1},y^2_{j-1})\leq 3$ and $d_X(y^1_{i-1},y^2_{j-1})\leq 3$, since the geodesic quadrangle with vertices $y^1_{i-1},y^1_i,y^2_j,y^2_{j-1}$ is contained in a single cone which is $\delta$-hyperbolic.

Let $v^1_{i-1}$ and $v^1_i$ be the end vertices of $R^1_i\cap P^1_{ess}$, and define $v^2_{j-1},v^2_j$ analagously.  By Lemma \ref{lem:Pesslocgeod}, $y^1_{i-1}$ and $y^1_i$ are each connected by an edge to $v^1_{i-1}$ and $v^1_i$, respectively, and similarly for $y^2_{j-1}$ and $y^2_j$.

Suppose $(\ast_a)$ holds (see Figure \ref{fig:daggera}).  If $v^1_{i-1}=v^2_{j-1}$, then $d_X(y^1_{i-1},y^2_{j-1})\leq 2$. Since $R^1_i$ and $R^2_j$ have exactly the same intersections with segments, and their intersections with essential paths are connected, if $v^1_{i-1}\neq v^2_{j-1}$ then they cannot both lie on a segment. It follows from the small cancellation assumption and the fact that $P^1_{ess}\cap R^1_i$ is connected that the only way $P^1_{ess}\cap R^1_i$ can contain $P^1_{ess}\cap \partial F$ for some face $F$ in $D$ is if $\partial F=R^1_i$.  Since $R^1_i$ contains an edge in a segment of $P^1_{ess}\cap P^2_{ess}$, this is impossible.  Thus there is a face $F$ in $D$ such that $v^1_{i-1}, v^2_{j-1}\in\partial F$.  
It follows that $d_X(v^1_{i-1},v^2_{j-1})\leq 1$, and $d_X(y^1_{i-1},y^2_{j-1})\leq 3$. The same reasoning proves $d_X(y^1_{i},y^2_{j})\leq 3$ as required.

Now suppose $(\ast_b)$ holds (see Figure \ref{fig:daggerb}). If $R^1_i=\partial F$, then the pairs $v^1_{i-1}, v^2_{j-1}$ and $v^1_i,v^2_j$ are either equal or connected by an internal arc in $D$. If $R^1_i\neq\partial F$, then  the pairs $v^1_{i-1}, v^2_{j-1}$ and $v^1_i,v^2_j$ are either equal or connected by a piece $\partial F\cap R^i$.  Thus, in either case, $d_X(y^1_{i-1},y^2_{j-1})\leq 3$ and $d_X(y^1_{i},y^2_{j})\leq 3$

\end{proof}

Theorem \ref{thm:thinconehyp} follows from Propositions \ref{hyppf:inotj} and \ref{hyppf:iandj}, and Lemma \ref{hyppf:selfintersections}.

\section{Acylindricity of actions on coned-off graphs}\label{sec:acylindricity}

In this section, we show that if $G$ is uniformly power-free, then $\TCG\subseteq \mathcal{AH}(G)$.  Recall that an action of a group $G$ by isometries on a metric space $Z$ is \textbf{acylindrical} if for all $\varepsilon>0$ there exist constants $M,N\geq 0$ such that for all $x,y\in Z$ with $d(x,y)\geq M$, the number of elements $g\in G$ satisfying $d(x,gx)\leq \varepsilon$ and $d(y,gy)\leq \varepsilon$ is at most $N$.  The proof of acylindricity will rely heavily on the following classification of essential quadrangles.

\subsection{Essential $S$-quadrangles}
\label{subsec:quads}
Let $Q_{X}=(\gamma^1,\gamma^2,\gamma^3,\gamma^4)$ be a geodesic quadrangle in $\Cay(G,X)$, so the terminal vertex of $\gamma^i$ is the initial vertex of $\gamma^{i+1}$ with indices considered modulo $4$. To each $\gamma^i$ associate an $S$-path $P^i$. We call $(P^1, P^2, P^3, P^4)$ an {\bf $S$--quadrangle} associated to $Q_{X}$.  We analogously define $Q_S=(P^1_{ess}, P^2_{ess}, P^3_{ess}, P^4_{ess})$ to be an {\bf essential $S$--quadrangle} associated to $Q_{X}$. We say $(Q^1,Q^2,Q^3,Q^4)$ is an essential $S$--quadrangle if it is an essential $S$-quadrangle associated to some geodesic quadrangle in $\Cay(G,X)$.  As for bigons, we append the superscript $t=1,2,3,4$ to any notation previously defined in Section \ref{sec:conedoffgraph}.

\begin{lemma} Let $Q$ be a simple closed path in an essential $S$-quadrangle $Q_S$. A reduced diagram $D$ with boundary $Q$ is a (possibly degenerate) combinatorial geodesic quadrangle.
\end{lemma}
\begin{proof}
This follows immediately from Proposition \ref{prop:combgeods}.
\end{proof}

We now recall some features of the classification of combinatorial geodesic quadrangles from \cite{ACGH2} which we will require to prove that the action of $G$ on $\Cay(G,X)$ is acylindrical. The key results we will use to limit the possibilities are Lemmas \ref{lem:essRconn}(ii) and \ref{lem:Pesslocgeod}.  Proving acylindricity relies on studying ``long, thin" quadrangles in $\Cay(G,X)$, which we make precise by requiring that
\begin{equation}\label{eq:longthin}
\min\set{l_{X}(\gamma^1),l_{ X}(\gamma^3)}\geq 3\max\set{l_{ X}(\gamma^2),l_{ X}(\gamma^4)}+25,
\end{equation} where $l_X(\alpha)$ denotes the length of the path $\alpha$ in $\Cay(G,X)$. 
All of the results in this section will be under the assumption that (\ref{eq:longthin}) holds.

We will need to use the notions of edge and face reductions in diagrams introduced in \cite[\S 3.2]{ACGH2}. We sketch the ideas here, and refer the reader to  \cite[\S 3.2]{ACGH2} for a more detailed discussion.   
\begin{itemize}
\item Given a diagram $D$ with an edge $e$ such that $D\setminus e$ is not connected, \textbf{reducing the edge} $e$ is the process of collapsing $e$ to a vertex to obtain a diagram $D'$, then removing this vertex and reattaching copies of it to each connected component of $D'\setminus e$ to obtain a collection of at least 2 diagrams.
\item Given a diagram $D$ with a face $F$ such that $D\setminus F$ is not connected, \textbf{reducing the face} $F$ is the process of first adding an edge $e$ to $F$ whose endpoints are on $\partial F\cap\partial D$ such that $D\setminus e$ is not connected, then reducing this edge.  Notice that in each of the new diagrams $D'$ formed, there is a face $F'$ coming from $F$.  \end{itemize}

Our goal for the rest of this subsection is the following theorem.

\begin{theorem}\label{thm:quadclass}
Let $Q_S=(P^1_{ess},P^2_{ess},P^3_{ess},P^4_{ess})$ be an essential $S$-quadrangle associated to a geodesic quadrangle $(\gamma^1,\gamma^2,\gamma^3,\gamma^4)$ satisfying $(\ref{eq:longthin})$. Either there exist $i,j$ such that $R^1_i=R^3_j$, or $P^1_{ess}\cap P^3_{ess}$ contains a path whose end vertices are at distance at least $\max\set{l_{ X}(\gamma^2),l_{ X}(\gamma^4)} + 6$ apart in $\Cay(G,X)$.
\end{theorem}

Let us collect a few basic observations about faces in (possibly degenerate) combinatorial geodesic quadrangles from Strebel's classification \cite{Str90} and \cite[\S 3]{ACGH2}:
\begin{lemma}\label{quadranglefacts} For a combinatorial geodesic quadrangle, the following hold.
\begin{enumerate}
 \item For each consecutive pair of sides, there is at most one face whose exterior boundary intersects both sides in edges and is contained in their union and whose interior degree at least 3. No such face has interior degree more than $4$.
 \item For each pair of opposite (non-consecutive) sides, there are at most two faces whose exterior boundary intersects both sides in edges and is contained in their union and whose interior degree at least 3.  There can only be one such face with interior degree at least 4, and there are no such faces with higher interior degree.
 \item If the exterior boundary of a face contains an edge and is contained in a single side, then the face has interior degree at most 6.
  \item If a face has no exterior boundary, then it has interior degree either 7 or 8.
\end{enumerate}
\end{lemma}
Notice that Lemmas \ref{numberoffacesbound} and \ref{lem:Pesslocgeod}, along with the small cancellation assumption, immediately rule out the last two possibilities for faces in diagrams whose boundaries are (contained in) essential $S$-quadrangles.

\begin{lemma}\label{lem:R2R4match} Let $(P^1_{ess}, P^2_{ess}, P^3_{ess}, P^4_{ess})$ be an essential $S$-quadrangle associated to a geodesic quadrangle $(\gamma^1,\gamma^2,\gamma^3,\gamma^4)$ satisfying $(\ref{eq:longthin})$. 
If there exist $r,s$ such that $R^2_r=R^4_s=R$, then there exist $i,j$ such that $R=R^1_i=R^3_j$.
\end{lemma}
\begin{proof}
Let $x^1,x^2$ be the end vertices of $\gamma^1$, where $x^2\in \gamma^2$. Suppose $R^2_r=R^4_s=R$ for some $r,s$. For $t=2,4$, define $Q^t$ to be the path obtained from the shortest subpath of $P^t_{ess}$ connecting some vertex $z^t$ on $P^2_r$ to $x^2$ (respectively $P^4_s$ to $x^1$).  

It is clear that $Q^t$ is embedded in $\Cay(G,S)$ and satisfies all the same conditions as an essential $S$-path. Since the chosen subpath has minimal length, either $z^2=x^2$ (respectively, $z^4=x^1$) or $z^2$ (respectively $z^4$) is also contained in another relator $R^2_{r'}$ (respectively $R^4_{s'}$). We either have $\abs{r-r'}=1$ or $z^2$ is the initial/terminal vertex of a self-intersection of type (i) in $P^2$.  In the former case, $\gamma^2$ and $z^2$ both lie on $R^2_{r}\cap R^2_{r'}$ and $d_{X}(\gamma^t,z^t)\leq 1$, while in the latter case, Lemma \ref{disttoPess} implies that $d_{X}(\gamma^2,z^2)\leq 1$. The same analysis can be used to show $d_{X}(\gamma^4,z^4)\leq 1$. Thus 
\begin{equation}\label{eq:z2z4}
d_{ X}(z^2,z^4)\geq l_{ X}(\gamma^1) - d_{ X}(x^2,z^2) - d_{ X}(z^4,x^1) \geq \max\set{l_X(\gamma^2),l_X(\gamma^4)} + 23,
\end{equation}
where the last inequality follows by (\ref{eq:longthin}).

Choose an $S$-geodesic $\gamma_S$ from $z^2$ to $z^4$. Since $z^2,z^4\in R$, $\gamma_S$ is necessarily contained in $R$. Consider the quadrangle $(\gamma_S, Q^2 ,P^1_{ess}, Q^4)$. Since $\gamma_S$ is contained in $R$, it is clear by construction that $Q^t$ intersects $\gamma_S$ only at the vertex $z^t$ for $t=2,4$.

\textbf{Case 1:}  $\gamma_S\cap P^1_{ess}=\emptyset$.  Let $Q$ be a simple cycle in the quadrangle $(\gamma_S, Q^2 ,P^1_{ess}, Q^4)$ containing $\gamma_S$. Let $D$ be a diagram with boundary $Q$. By construction, $R$ intersects $Q^2$ and $Q^4$ only at $z^2$ and $z^4$, respectively.  Thus, if there is a face $F\subset D$ with $\partial F=R$, then since our small cancellation assumption rules out case (iii) above, the exterior boundary of $F$ must be contained in $\gamma_S$ and $P^1_{ess}$ (see Figure \ref{fig:Lem68}(A)).   Therefore $F$ must belong to case (ii) above and satisfy $i(F)\leq 4$, and so the end vertices of $P^1_{ess}\cap R$ are at $X$--distance at most $1$ from $z^2$ and $z^4$ respectively.  In this case, it follows from Lemma \ref{lem:essRconn}(ii) and ($\ref{eq:z2z4}$) that $R=R^1_i$ for some $i$.

If no face in $D$ satisfies $\partial F=R$, then for $t=2,4$, there is no face whose  exterior boundary contains edges in both $\gamma_S$ and $Q^t$ and is contained in their union.  To see this, note that if $F$ was such a face, then $i(F)\leq 4$ by (i) above, so $\partial F$ is a union of $\partial F\cap Q^t$ and at most 5 pieces (the fifth being $\partial F\cap \gamma_S=\partial F\cap R$). Applying Lemmas \ref{numberoffacesbound} and \ref{lem:Pesslocgeod}, we see that $\partial F$ is either contained in the union of at most 11 pieces or a geodesic and the union of at most 11 pieces, both of which contradict the small cancellation assumption. By a similar argument, no face can have its exterior boundary contained in $\gamma_S$ and $P^1_{ess}$. It follows that any face whose boundary intersects $\gamma_S$ in an edge must intersect at least two additional sides of $Q$ in an edge.  There can be at most two such faces, and so $\gamma_S$ is contained in a union of at most two pieces.  Thus $d_{ X}(z^2,z^4)\leq 1$, which contradicts (\ref{eq:z2z4}).

\textbf{Case 2:}  $\gamma_S\cap P^1_{ess}\neq\emptyset$. If $z^2,z^4\in \gamma_S\cap P^1_{ess}$, then it follows from Lemma \ref{lem:essRconn}(ii) and ($\ref{eq:z2z4}$) that $R=R^1_i$ for some $i$, so suppose that $z^2\not\in \gamma_S\cap P^1_{ess}$.  Let $Q$ be a simple cycle in the quadrangle $(\gamma_S, Q^2 ,P^1_{ess}, Q^4)$ containing $z^2$.  It is clear that $Q$ is contained in $\gamma^S\cup P^1_{ess}\cup Q^2$. No face in a diagram with this boundary can have its exterior boundary contained in only one side, or in a pair of sides if one of those is $\gamma_S$.  Indeed, if such a face $F$ existed, then, as above, $\partial F$ must be contained in a union of either at most 11 pieces, or 11 pieces and a geodesic, both of which contradict the small cancellation assumption. It follows that the face $F$ containing $z^2$ must intersect $P^1_{ess}$ (see Figure \ref{fig:Lem68}(B)).

\begin{figure}[h!] 
\centering
\subfloat[]{
\resizebox{3in}{!}{
  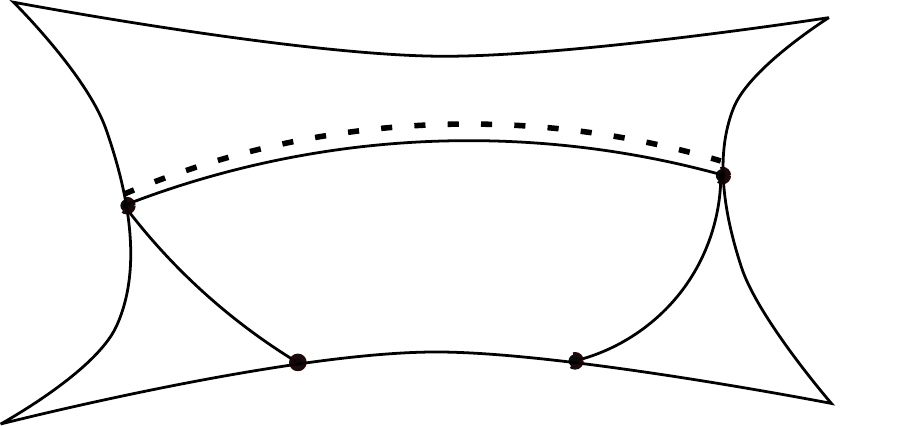}
}
\subfloat[]{
  \resizebox{3in}{!}{
  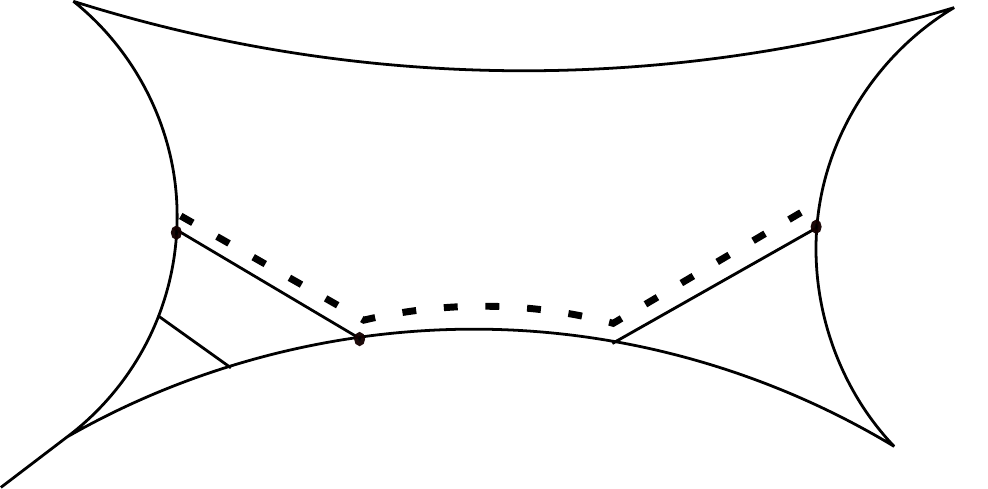}
}
\caption{(A) Proof of Lemma 6.8 in the case where $\gamma_S\cap P^1_{ess}=\emptyset$ and there is a face $F$ such that $\partial F=R$.  (B) Proof of Lemma 6.8 in the case where $\gamma_S\cap P^1_{ess}\neq\emptyset$.  The dotted path is $\gamma_S$, and the boundary of the shaded region is $R$.}\label{fig:Lem68} 
\end{figure}

Thus $P^1_{ess}\cap (R\cap\partial F)\neq\emptyset$, and hence $d_{ X}(z^2,P^1)\leq 1$. Applying the same argument to $z^4$ if necessary, it follows from Lemma \ref{lem:essRconn}(ii) and ($\ref{eq:z2z4}$) that $R=R^1_i$ for some $i$. 

Applying this argument to $P^3_{ess}$ instead of $P^1_{ess}$, it follows that $R=R^3_j$ for some $j$. This completes the proof.
\end{proof}

\begin{lemma} \label{lem:bdgammatoPess}
Let $(P^1_{ess},P^2_{ess},P^3_{ess},P^4_{ess})$ be an essential $S$--quadrangle associated to a geodesic quadrangle $(\gamma^1,\gamma^2,\gamma^3,\gamma^4)$.  Let $Q$ be a simple cycle in the quadrangle, and let $D$ be a diagram with boundary $Q$.  Suppose $F$ is a face of $D$ such that $\partial F$ shares an edge with at most 2 essential paths, $P^t_{ess}$ and $P^{t'}_{ess}$, and $i(F)\leq 4$. If $z$ is an end vertex of $\partial F\cap P^r_{ess}$ with $r\in\{t,t'\}$, then $d_{ X}(z,\gamma^r)\leq 3$.
\end{lemma}

\begin{proof}
Suppose $\partial F$ shares an edge with $P^t_{ess}$, and let $z$ be an end vertex of $P^t_{ess}\cap\partial F$.  If there exists $i$ such that $\partial F=R^t_i$, then the result follows by Lemma \ref{disttoPess}.  Otherwise, $\partial F\cap P^t_{ess}$ is contained in $1\leq s\leq 5$ pieces $R^t_i\cap \partial F$ by Lemma \ref{numberoffacesbound}.  

If $s\geq 2$, then there is a point $y\in \gamma^t$ that is connected to $\partial F\cap P^t_{ess}$ by an edge in $\Cay(G,X)$, and thus $d_{ X}(y,z)\leq 1+\left\lceil\frac54\right\rceil=3$.

To complete the proof, we will show that the case $s=1$ is not possible. First, if $\partial F\cap P^{t'}_{ess}$ does not contain an edge, then $\partial F$ is contained in the union of at most $5$ pieces (at most 4 interior pieces and the exterior piece $\partial F\cap P^t_{ess}$), which contradicts the small cancellation assumption.  Thus, $\partial F\cap P^{t'}_{ess} $ must contain an edge.  If $\partial F=R^{t'}_j$ for some $j$, then $\partial F\setminus (\partial F\cap P^{t'}_{ess})$ is contained in the union of at most 5 pieces, and so $|\partial F\cap P^{t'}_{ess}|> \frac{19}{24}|\partial F|$, which contradicts Lemma \ref{lem:Pesslocgeod}.  If $\partial F\neq R^{t'}_j$ for all $j$, then $\partial F\cap P^{t'}_{ess}$ is contained in the union of at most 5 pieces, by Lemma \ref{numberoffacesbound}, and so $\partial F$ is contained in the union of at most $10$ pieces, contradicting the small cancellation bound.
\end{proof}

We are now ready to prove Theorem \ref{thm:quadclass} in two steps (Propositions \ref{prop:degenquads} and \ref{prop:longthin}).  We begin with the ``degenerate" case.

\begin{proposition}\label{prop:degenquads} If $Q_{ X}=(\gamma^1,\gamma^2,\gamma^3,\gamma^4)$ satisfies $(\ref{eq:longthin})$ and the essential quadrangle $Q_S=(P^1_{ess}, P^2_{ess}, P^3_{ess}, P^4_{ess})$ associated to $Q_{ X}$ does not admit a simple cycle $Q$ which intersects each $P^t_{ess}$ in an edge, then one of the following occurs:
	\begin{enumerate}
	\item there exist $i,j$ such that $R^1_i=R^3_j$; or 
	\item $P^1_{ess}\cap P^3_{ess}$ contains a subpath whose end vertices are at $ X$--distance at most \[\max\set{l_{ X}(\gamma^2),l_{ X}(\gamma^4)} + 3\] from the end vertices of $\gamma^1$.
	\end{enumerate} 
\end{proposition}

\begin{proof}
We will assume that (i) does not hold, and deduce that (ii) does. The proof is in 4 steps:
\begin{enumerate}
 \item[a)] Prove $P^1_{ess}\cap P^3_{ess}\neq\emptyset$.
 \item[b)] Prove $P^1_{ess}\cap P^3_{ess}$ is connected.
 \item[c)] Let $F$ be any face in a disc diagram whose boundary intersects both $P^1_{ess}$ and $P^3_{ess}$ in an edge.  Then prove one of the following occurs: $F$ intersects one of $P^2_{ess}$ or $P^4_{ess}$ in an edge; or $F$ contains a point in $P^1_{ess}\cap P^3_{ess}$ and satisfies $e(F)=1$, $i(F)=2$, and $D$ is a combinatorial geodesic triangle of type $III_1$ or $V$.
 \item[d)] Deduce that (ii) holds.
\end{enumerate} 
\textbf{Step a)} Note that by Lemma \ref{lem:R2R4match}, if $R^1_i\neq R^3_j$ for all $i,j$ then $R^2_k\neq R^4_{k'}$ for all $k,k'$, as well.  As in the proof of Lemma \ref{lem:bdgammatoPess}, let $x$ be the initial endpoint of $\gamma^1$, so that $x\in \gamma^1\cap \gamma^2$.

Since $Q_S$ does not admit a simple cycle which intersects each $P^t_{ess}$ in an edge, $P^t_{ess}\cap P^{t'}_{ess}\neq \emptyset$ for either $(t,t')=(1,3)$ or $(t,t')=(2,4)$.  To complete step a) we show that $(t,t')=(1,3)$.

Suppose $P_{ess}^2\cap P_{ess}^4\neq \emptyset$, and let $v$  be the vertex on $P^2_{ess}\cap P^4_{ess}$  closest to  $x$ along $P^2_{ess}$.  Let $e$ be the edge in $P^2_{ess}\cap P^4_{ess}$ containing $v$ as an end vertex, and choose relators $R^2_i$ and $R^4_j$ containing $e$.  By assumption, $R^2_i\neq R^4_{j'}$ for any $j'$, so by Lemma \ref{numberoffacesbound}, $P_{ess}^4\cap R^2_i$ is contained in a union of at most 5 pieces.  By Lemma \ref{disttoPess}, there is an edge connecting a vertex in $\gamma^4$ to a vertex in one of these pieces, hence $d_{ X}(\gamma^4,v)\leq 1+\left\lceil\frac54\right\rceil=3$.  Similarly, $d_{ X}(\gamma^2,v)\leq 3$.  Therefore, $d_{ X}(\gamma^2,\gamma^4)\leq 6$, which contradicts (\ref{eq:longthin}).  Therefore, $P^2_{ess}\cap P^4_{ess}=\emptyset$ and $P^1_{ess}\cap P^3_{ess}\neq\emptyset$.

\textbf{Step b)} If $P^1_{ess}\cap P^3_{ess}$ is not connected, then there is a disc diagram $D$ whose boundary is contained in $P^1_{ess}\cap P^3_{ess}$. Recall that $D$ must be a diagram of type $I_1$ in Strebel's classification. Let $F$ be a face with $e(F)=1$ and $i(F)\leq 1$.  If there is no $i$ such that $\partial F=R^1_i$, then by Lemmas \ref{numberoffacesbound} and \ref{lem:Pesslocgeod}
\[
 \abs{\partial F} < \left(\frac{5}{24} + \frac{1}{24} + \frac34\right) \abs{\partial F},
\]
which is a contradiction. Similarly there is some $j$ such that $\partial F=R^3_j$ contradicting our assumption that (i) fails.

\textbf{Step c)} Now suppose $D$ is a disc diagram whose boundary is contained in $P^1_{ess}\cup P^2_{ess}\cup P^3_{ess}$ and contains an edge in each. Let $F$ be the face containing the vertex in $P^1_{ess}\cap P^3_{ess}\cap\partial D$. If $\partial F$ does not contain an edge in $P^2_{ess}$ then (as in step b)) either $\partial F=R^1_i=R^3_j$ for some $i,j$ (which is a contradiction), or
\[
 \abs{\partial F} < \abs{\partial F}\left(\frac{5}{24} +\frac{18}{24} + \frac{i(F)}{24}\right),
\]
and so $i(F)\geq 2$.  Thus $D$ is of type $III_1$, $IV$ or $V$, and the boundaries of all other faces in $D$ intersect $P^2_{ess}$ in an edge. To see that $D$ is not of type $IV$ note that in a type $IV$ diagram there is a face $F'$ satisfying $e(F')=1$ and $i(F')=4$ whose external boundary is contained in one of the $P^t_{ess}$, but no such faces can exist in an essential $S$--triangle (see the remark after Lemma \ref{quadranglefacts}).

\textbf{Step d)} Let $x,y$ be the end vertices of $P^1_{ess}\cap P^3_{ess}$ where $x$ is closest to the initial vertex $v^1$ of $P^1_{ess}$. Since we are assuming (\ref{eq:longthin}) it suffices to prove that $d_X(x,\gamma^2)\leq 11$ and $d_X(y,\gamma^4)\leq 11$.  Either $x\in P^2_{ess}$ or there exists a disc diagram $D$ as described in step c). 

\textbf{Case d)(i):} $x\in P^2_{ess}$. If $x\in P^2_{ess}$ and the two edges in $P^2_{ess}$ with $x$ as an end vertex do not lie in a common $R^2_k$, then $d_X(x,\gamma^2)\leq 1$ by Lemma \ref{disttoPess}. If they do lie in a common $R^2_k$, then either $R^2_k$ is not equal to any $R^1_i$ or not equal to any $R^3_j$, and without loss of generality we may assume the former is true. Let $Q^2$ denote the subpath of $P^2_{ess}$ connecting $v^1$ to $x$. If $R^2_k\cap Q^2 \subseteq P^1_{ess}$ then there is a path from $x$ to the other end vertex of $R^2_k\cap Q^2$ which is a union of at most $5$ pieces, so by Lemma \ref{disttoPess}
\[
 d_X(x,\gamma^2) \leq \left\lceil \frac54 \right\rceil + 1 =3.
\]
Otherwise there is a combinatorial geodesic bigon $B$ whose boundary is contained in $Q^2\cup P^1_{ess}$. Let $F$ be the face in $B$ whose boundary contains the point on $Q^2$ closest to $x$ along $P^2_{ess}$, so $e(F)=i(F)=1$ and $\partial F$ contains an edge in $Q^2\cap R^2_k$. If $\partial F\neq R^2_k$ then $\partial F$ is a union of at most $5$ pieces (coming from its intersection with $Q^2$) one additional piece from its internal boundary and a path of length less than $\frac34\abs{\partial F}$ from its intersection with $P^1_{ess}$, which contradicts the small cancellation assumption. For the same reasoning there must be some $i$ such that $\partial F=R^1_i$ contradicting the assumption that $R^2_k$ is not equal to any $R^1_i$.

\textbf{Case d)(ii):} $x\not\in P^2_{ess}$. Let $F$ be the face in $D$ whose boundary contains $x$. Without loss of generality we may assume that there is no $i$ such that $\partial F=R^1_i$, so $\partial F\cap P^1_{ess}$ is contained in the union of at most $5$ pieces. When $D$ is not of type $V$ there is a path of length at most $2$ connecting $x$ to one of the end vertices $x'$ of $\partial F\cap P^2_{ess}$.  When $D$ is of type $V$ there is a face $F'$ neighbouring $F$ and a path of length at most $2$ connecting $x$ to one of the end vertices $x'$ of $\partial F'\cap P^2_{ess}$. If $\partial F$ (respectively $\partial F'$) is equal to one of the $R^2_k$, then by Lemma \ref{disttoPess}, $d_X(x,\gamma^2)\leq 3$. Otherwise, pick $k$ such that $x'\in R^2_k$ and notice that by Figure \ref{fig:treecases} there is a path from some vertex in $\gamma^2$ to $x'$ which is a union of at most three pieces in $R^2_k$. Thus $d_X(\gamma^2,x)\leq 3$. Similarly, we may deduce that $d_X(\gamma^4,y)\leq 3$.
\end{proof}

In what follows, we say a vertex of a quadrangle $Q=(Q^1_{ess}, Q^2_{ess}, Q^3_{ess}, Q^4_{ess})$ is {\bf distinguished} if it is the end vertex of some $Q^j_{ess}$.  In a reduced diagram $D$ with boundary $Q$, we say a face $F$ of $D$ is distinguished if there is a distinguished vertex in the interior of its exterior boundary. Let us first rule out two scenarios using the assumption $(\ref{eq:longthin})$.

\begin{lemma}\label{lem:noladder} Suppose $Q_S=(P^1_{ess}, P^2_{ess}, P^3_{ess}, P^4_{ess})$ is an essential $S$--quadrangle associated to $(\gamma^1,\gamma^2, \gamma^3, \gamma^4)$ and $(\ref{eq:longthin})$ holds.  If there is a diagram $D$ with boundary $Q_S$ admitting an internal arc connecting $P^2_{ess}$ to $P^4_{ess}$, then there exist $i,j$ such that $R^1_i=R^3_j$.
\end{lemma}
\begin{proof} Suppose for a contradiction that $D$ is such a diagram containing an internal arc $\alpha$ and that for all $i,j$, $R^1_i\neq R^3_j$. Let $F^1,F^3$ be the faces in $D$ such that $\alpha=\partial F^1\cap\partial F^3$, and for $t=2,4$ let $z^t$ be the unique vertex in $P^t_{ess}\cap\alpha$. We claim that $d_X(\gamma^t,z^t)\leq 5$ for $t=2,4$.  Once this is verified we obtain a contradiction to $(\ref{eq:longthin})$, since 
\[
d_X(\gamma^2,\gamma^4)\leq d_X(\gamma^2,z^2)+d_X(z^2,z^4)+d_X(z^4,\gamma^4)\leq 11.
\]
If $e(F)=2$, then by Lemma \ref{lem:bdgammatoPess} we have $d_X(\gamma^t,z^t)\leq 3$ for $t=2,4$, and so we may assume $\partial F^s$ contains an edge of $P^s_{ess}$ for $s=1,3$. If there exists some $r$ such that $R^2_r\cap P^2_{ess}$ has an end vertex in $\partial F^s$, then $d_X(\gamma^2,z^2)\geq 1$ by Lemma \ref{disttoPess}, so we may also assume this does not happen.  Under these assumptions, $(\partial F^1\cap\partial F^3)\cap P^2_{ess}$ is strictly contained in some $R^2_r$.

Suppose there is another face $F\not\in\set{F^1,F^3}$ such that $\partial F\cap R^2_r$ contains an edge. If the boundary of this face contains an end vertex $v$ of $R^2_r\cap P^2_{ess}$, then there is an edge connecting $\gamma^2$ to $v$ by Lemma \ref{disttoPess}, and since there is a path from $v$ to $z^2$ which is contained in a union of two pieces in $R^2_r$, it follows that $vz^2$ is an edge. Thus $d_X(\gamma^2,z^2)\leq 2$. Otherwise, since $\partial F\neq R^2_r$ and $e(F)=i(F)\in\set{1,2}$, $\partial F$ is the union of $P^3_{ess}\cap \partial F$ and at most 3 pieces, contradicting the small cancellation assumption.

Hence we may now assume that $R^2_r\cap P^2_{ess}$ consists of two subpaths $N^s$ of $P^2_{ess}\cap P^s_{ess}$ for $s=1,3$ and the pieces $R^2_r\cap F^t$ for $t=2,4$. If both of the $N^s$ have diameter at least $3$ in $\Cay(G,X)$, then it follows from Lemma \ref{lem:essRconn}(ii) that $R^2_r=R^1_i=R^3_j$ for some $i,j$, which is a contradiction. If, without loss of generality, $N^1$ has diameter at most $2$, then there is a path from $\gamma^2$ to $z^2$ of length at most $4$ in $X$ consisting of a path of length at most 2 (from $\gamma^2$ to an end vertex of $R^2_r\cap P^2_{ess}$ contained in $N^1$), a path of length at most $2$ to a vertex in $\partial F'\cap N^1$, and an edge connecting this vertex to $z^2$.
\end{proof}

\begin{lemma}\label{lem:nozipper} If $Q_S=(P^1_{ess}, P^2_{ess}, P^3_{ess}, P^4_{ess})$ is an essential $S$--quadrangle associated to $(\gamma^1,\gamma^2, \gamma^3, \gamma^4)$ and $(\ref{eq:longthin})$ holds, then no diagram whose boundary is contained in $Q_S$ has the form of Figure $\ref{fig:zipper}$ below.
\end{lemma}
\begin{figure}[H]

\begin{tikzpicture}[yscale=1.5,xscale=1, 
vertex/.style={draw,fill,circle,inner sep=0.3mm}]

\draw[thick] (3,1)--(-3,1)--(-2.5,0.5)--(-2.5,-0.5)--(-3,-1)--(3,-1)--(2.5,-0.5)--(2.5,0.5)--(3,1);

\draw[very thin, dotted] 	(-2.5,0.5) -- (-2.2,1)
					(-2.5,-0.5) -- (-2.2,-1)
					(-2.625,0.625) -- (-2.4,1)
					(-2.625,-0.625) -- (-2.4,-1)
					(-2.75,0.75) -- (-2.6,1)
					(-2.75,-0.75) -- (-2.6,-1)
					(2.5,0.5) -- (2.2,1)
					(2.5,-0.5) -- (2.2,-1)
					(2.625,0.625) -- (2.4,1)
					(2.625,-0.625) -- (2.4,-1)
					(2.75,0.75) -- (2.6,1)
					(2.75,-0.75) -- (2.6,-1);
					
\draw[very thin]	(-2.5,0.2) -- (-0.7,0.2) -- (0.7,-0.2) -- (0.7,-1)
					(-0.7,0.2) -- (-0.7,1)
					(0.7,-0.2) -- (2.5,-0.2);

\path (-1.2,-1) node[below] {$Q_S\cap P^1_{ess}$};
\path (-2.7,0) node[left] {$Q_S\cap P^2_{ess}$};
\path (0,1) node[above] {$Q_S\cap P^3_{ess}$};
\path (2.7,0) node[right] {$Q_S\cap P^4_{ess}$};

\node[vertex] (a1) at (-2.5,0.2) {};
\path (a1) node[below right] {$z^2$};

\node[vertex] (a2) at (2.5,-0.2) {};
\path (a2) node[above left] {$z^4$};

\path (-1.2,-0.5) node[] {$B^1$};
\path (-1.6,0.6) node[] {$B^2$};
\path (1.2,0.5) node[] {$B^3$};
\path (1.6,-0.6) node[] {$B^4$};

\end{tikzpicture}

\caption{A diagram containing a zipper of length $0$ whose ends are of type $1$. Optional internal arcs are indicated by dotted lines.}\label{fig:zipper}
\end{figure}
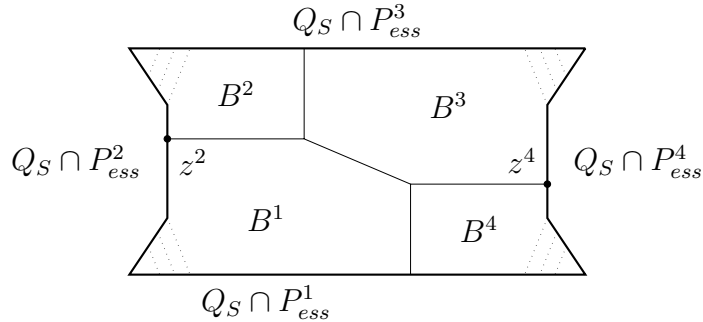
\begin{proof} Suppose for a contradiction that $D$ is such a diagram.  Then $D$ has four faces $B^t$ such that $\partial B^t\cap D'$ intersects both $P^t_{ess}$ and $P^{t+1}_{ess}$. We now show that (\ref{eq:longthin}) excludes this possibility.

Let $z^2$ be the vertex in $P^2_{ess}\cap\partial B^1\cap \partial B^2$ and let $z^4$ be the vertex in $P^4_{ess}\cap\partial B^3\cap \partial B^4$.   There is a path from $z^2$ to $z^4$ in $\Gamma(G,X)$ which is a union of at most 3 pieces in $\partial B^1$ and $\partial B^4$, so $d_{ X}(z^2,z^4)\leq 2$.  By Lemma \ref{lem:bdgammatoPess}, there are points $y^2\in \gamma^2$ and $y^4\in\gamma^4$ such that $d_{ X}(y^t,z^t)\leq 3$ for $t=2,4$, and thus $d_{ X}(y^2,y^4)\leq 8$. Therefore,
\[
 l_{{X}}(\gamma_1) \leq d_{ X}(x,y^2)+d_{ X}(y^2,y^4) + d_{ X}(y^4,x') \leq 2\max\set{l_{ X}(\gamma^2),l_{ X}(\gamma^4)}+8, 
\]
which contradicts the assumption (\ref{eq:longthin}).
\end{proof}

\begin{proposition}\label{prop:essentialquads} If the essential quadrangle $(P^1_{ess}, P^2_{ess}, P^3_{ess}, P^4_{ess})$ associated to $(\gamma^1,\gamma^2, \gamma^3, \gamma^4)$ admits a simple cycle $Q$ which intersects each $P^t_{ess}$ in at least an edge and $(\ref{eq:longthin})$ holds, then any diagram $D$ with boundary $Q$ and no internal arc connecting $P^2_{ess}$ to $P^4_{ess}$ has one of the forms given in Figure $\ref{fig:reduciblequads}$.
\end{proposition}

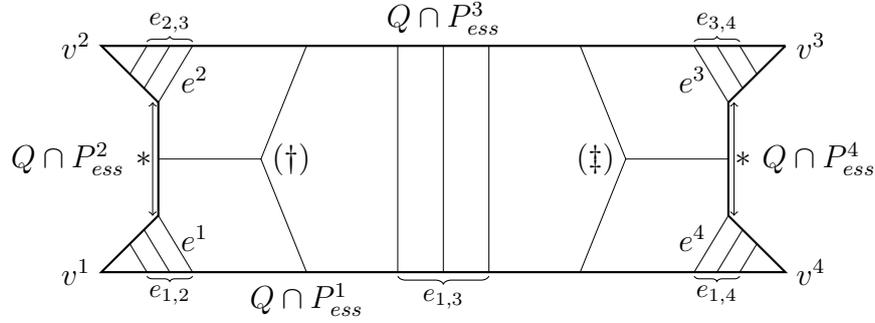
\begin{figure}[H]
\begin{tikzpicture}[yscale=1.5,xscale=1.5, 
vertex/.style={draw,fill,circle,inner sep=0.3mm}]

\draw[thick] (3,1)--(-3,1)--(-2.5,0.5)--(-2.5,-0.5)--(-3,-1)--(3,-1)--(2.5,-0.5)--(2.5,0.5)--(3,1);

\draw[very thin] 	(-2.5,0.5) -- (-2.2,1)
					(-2.5,-0.5) -- (-2.2,-1)
					(-2.625,0.625) -- (-2.4,1)
					(-2.625,-0.625) -- (-2.4,-1)
					(-2.75,0.75) -- (-2.6,1)
					(-2.75,-0.75) -- (-2.6,-1)
					(-2.5,0) -- (-1.6,0) -- (-1.2,1)
					(-1.6,0) -- (-1.2,-1)
					(-0.4,-1) -- (-0.4,1)
					(0,-1) -- (0,1)
					(0.4,-1) -- (0.4,1)
					(2.5,0.5) -- (2.2,1)
					(2.5,-0.5) -- (2.2,-1)
					(2.625,0.625) -- (2.4,1)
					(2.625,-0.625) -- (2.4,-1)
					(2.75,0.75) -- (2.6,1)
					(2.75,-0.75) -- (2.6,-1)
					(2.5,0) -- (1.6,0) -- (1.2,1)
					(1.6,0) -- (1.2,-1);

\path (-1.6,0) node[right] {($\dagger$)};
\path (1.6,0) node[left] {($\ddagger$)};

\path (-1.2,-1) node[below] {$Q\cap P^1_{ess}$};
\path (-2.7,0) node[left] {$Q\cap P^2_{ess}$};
\path (0,1) node[above] {$Q\cap P^3_{ess}$};
\path (2.7,0) node[right] {$Q\cap P^4_{ess}$};

\draw[very thin, <->]	(-2.55,-0.5) -- (-2.55,0.5);
\draw[very thin, <->]	(2.55,-0.5) -- (2.55,0.5);

\path (-2.63,0) node[] {$\ast$};
\path (2.63,0) node[] {$\ast$};

\path (-3,-1) node[left] {$v^1$};
\path (-3,1) node[left] {$v^2$};
\path (3,1) node[right] {$v^3$};
\path (3,-1) node[right] {$v^4$};

\path (-2.4,-0.7) node[right] {$e^1$};
\path (-2.4,0.7) node[right] {$e^2$};
\path (2.4,0.7) node[left] {$e^3$};
\path (2.4,-0.7) node[left] {$e^4$};

\draw [decorate,decoration={brace,amplitude=2pt},xshift=0pt,yshift=-1pt] (-2.2,-1) -- (-2.6,-1) node [black,midway,yshift=-0.25cm] {\footnotesize $e_{1,2}$};

\draw [decorate,decoration={brace,amplitude=2pt},xshift=0pt,yshift=1pt] (-2.6,1) -- (-2.2,1) node [black,midway,yshift=0.25cm] {\footnotesize $e_{2,3}$};

\draw [decorate,decoration={brace,amplitude=2pt},xshift=0pt,yshift=1pt] (2.2,1) -- (2.6,1) node [black,midway,yshift=0.25cm] {\footnotesize $e_{3,4}$};

\draw [decorate,decoration={brace,amplitude=2pt},xshift=0pt,yshift=-1pt] (2.6,-1) -- (2.2,-1) node [black,midway,yshift=-0.25cm] {\footnotesize $e_{1,4}$};

\draw [decorate,decoration={brace,amplitude=3pt},xshift=0pt,yshift=-1pt] (0.4,-1) -- (-0.4,-1) node [black,midway,yshift=-0.3cm] {\footnotesize $e_{1,3}$};

\end{tikzpicture}

\caption{Diagrams with boundary $Q$ have the following form: the tripods marked $(\dagger)$ and $(\ddagger)$ are optional, and $e_{t,t'}$ denotes the number of interior arcs with end vertices in $P^t_{ess}$ and $P^{t'}_{ess}$. When the tripod does not appear, the segments marked $\ast$ may have length 0.}\label{fig:reduciblequads}
\end{figure}

\begin{proof}
Using Lemma \ref{quadranglefacts} we see that every face of $Q$ has external boundary which intersects at least two of the sides $P^t_{ess}$. 

For $1\leq t<t'\leq 4$, define $E_{t,t'}$ to be the set of edges in $D$ which have one end vertex in $P^t_{ess}$ and the other in $P^{t'}_{ess}$, and let $e_{t,t'}=\abs{E_{t,t'}}$. By assumption $e_{2,4}=0$. Notice that no two edges in $E_{t,t'}$ can have an end vertex in common.  Indeed, if this were the case, then there would be a face $F\subset D$ whose boundary is contained in at most 2 pieces and its intersection with $P^t_{ess}$ (or $P^{t'}_{ess}$), which contradicts Lemma \ref{lem:Pesslocgeod}.

For each $t$ (considered modulo $4$) let $v^t$ be the vertex in $Q\cap P^t_{ess}\cap P^{t+1}_{ess}$, and whenever $E_{t,t+1}\neq \emptyset$, choose $e^t\in E_{t,t+1}$ such that the end vertex of $e^t$ on $P^t_{ess}$ is furthest from $v^t$ along $P^t_{ess}$. 

{\bf Case 1.}  If there is an interior arc $\alpha$ connecting $P^1_{ess}$ to $P^3_{ess}$ (i.e., $e_{1,3}>0$) consider the two diagrams $D_2,D_4$ obtained by removing the closure of $\alpha$ in $D$ and then reattaching a copy of it to each connected component of $D\setminus\overline{\alpha}$. The two resulting diagrams $D_t$ are combinatorial geodesic triangles with sides $P^t_{ess}\cap Q$, $P^1_{ess}\cap \partial D^t$ and the union of $P^3_{ess}\cap \partial D^t$ and $\alpha$. Using Strebel's classification, the triangle $D^2$ is of type: $I_2$ if and only if $e_{1,2}=e_{2,3}=0$; $I_3$, $II$ or $III_2$ if and only if at least one of $e_{1,2},e_{2,3}\neq 0$ and $(\dagger)$ does not appear (with $III_2$ happening if and only if the segment marked $*$ has length $0$); and $V$ if and only if $(\dagger)$ does appear. Recall that the triangle cannot be of type $IV$ by Lemma \ref{quadranglefacts}(iii). A similar analysis can be done for $D^4$.

{\bf Case 2.}  Suppose that $e_{1,3}=0$. Let us prove that there is a face whose boundary intersects both $P^1_{ess}$ and $P^3_{ess}$ in edges. If this is the case then we can add an internal arc connecting $P^1_{ess}$ to $P^3_{ess}$ to $D$ and obtain a new combinatorial geodesic quadrangle $D''$ satisfying all the properties of Lemma \ref{quadranglefacts}, so we may follow the same analysis as in Case 1.

Suppose for a contradiction that no face intersects both $P^1_{ess}$ and $P^3_{ess}$ in edges. Let $D'$ be the subdiagram of $D$ whose boundary consists of $e^t$ and the subpath of $Q\cap P^t_{ess}$ between $e^{t-1}\cap P^t_{ess}$ and $e^t\cap P^t_{ess}$ for each $t$.  By construction, $D'$ does not contain $v^t$ for any $t$ such that $e_{t,t+1}>0$, and no pair of consecutive sides in $D'$ are connected by an edge (see Figure \ref{fig:reduciblequads}). The diagram $D'$ is special in the sense of \cite[Definition $3.14$]{ACGH2}, but it cannot be extraordinary, as all such diagrams have faces which intersect only one side of the quadrangle and this is prohibited in our case by Lemma \ref{quadranglefacts}(iii). Hence, $D'$ is a zipper, and this zipper has length $0$ and both ends are of type $1$ (any other zipper has a face which intersects only one side of the quadrangle). But such a diagram is not possible by Lemma \ref{lem:nozipper}.
\end{proof}

Notice that the assumption (\ref{eq:longthin}) is only used in the above proof when $e_{1,3}=e_{2,4}=0$. We are now ready to deal with the ``non-degenerate'' case of Theorem \ref{thm:quadclass}.

\begin{proposition}\label{prop:longthin} If an essential quadrangle $(P^1_{ess},P^2_{ess},P^3_{ess},P^4_{ess})$ associated to a geodesic quadrangle $Q_S = (\gamma^1,\gamma^2,\gamma^3,\gamma^4)$ satisfying $(\ref{eq:longthin})$ admits a simple cycle $Q$ which intersects each $P^t_{ess}$ in an edge, then there exist $i,j$ such that $R^1_i=R^3_j$.
\end{proposition}

\begin{proof} If $e_{2,4}\neq 0$ then the conclusion holds by Lemma \ref{lem:noladder}, so we may apply Proposition \ref{prop:essentialquads}. 

Let us assume for a contradiction that $R^1_i\neq R^3_j$ for all $i,j$. Now consider a diagram $D$ with boundary $Q$, which necessarily has one of the forms given in Figure \ref{fig:reduciblequads}. 

\textbf{Case 1:}  $e_{1,3}\geq 3$. In this case, there are two faces $F,F'\subset D$ whose boundaries intersect in an internal arc in $D$ such that $e(F)=i(F)=e(F')=i(F')=2$ and the exterior boundary of each face has one connected component in $P^1_{ess}$ and another in $P^3_{ess}$. If $\partial F$ or $\partial F'$ is not one of the $R^1_i$ or $R^3_j$, then its boundary is a union of at most 12 pieces by Lemma \ref{numberoffacesbound}, which is a contradiction. If $\partial F=R^1_i$ and $\partial F'=R^1_{i'}$ for some $i,i'$, then let $\alpha$ be the subpath $P^3_{ess}\cap (\partial F\cup \partial F')$ and define $J$ to be the set of indices such that $R^3_j$ contains an edge in $\alpha$. Let $j_1,j_2$ be the minimal (resp. maximal) elements of $J$. By Lemma \ref{disttoPess}, $y^3_{j_1}$ and $y^3_{j_2-1}$ are connected to $\alpha$ by edges. Now $\alpha$ is contained in a union of at most 5 pieces in $\partial F$ and at most 5 pieces in $\partial F'$, hence
\[
 (j_2-1)-j_1 \leq d_X(y^3_{j_1},y^3_{j_2-1}) \leq 2 + 2\left\lceil\frac54\right\rceil = 6.
\]
Thus, in total, $\alpha$ is contained in a union of at most $9$ pieces (since one element of $J$ could contribute a piece in each of $\partial F$ and $\partial F'$), hence one of $\partial F$ or $\partial F'$ is a union of at most $4$ pieces which intersect $\alpha$, two internal arcs in $D$, and $P^1_{ess}\cap \alpha$.  Therefore, $\abs{\partial F} < \left(\frac34+\frac{6}{24}\right)\abs{\partial F}$, which is a contradiction.

Hence we may assume $\partial F=R^1_i$ and $\partial F'=R^3_j$ for some $i,j$. Now let $J'$ be the set of $j'$ such that $\partial F\cap P^3_{ess}$ contains an edge in $R^3_{j'}$, and let $j'_1$ and $j'_2$ be the minimal/maximal elements of $J'$.   Note that $j$ is not in the interval $j'_1,\ldots,j'_2$. Assume $j<j'_1$; the other possibility can be handled in the same way. There are edges connecting $y^3_j$ and $y^3_{j'_2-1}$ to $\partial F\cap P^3_{ess}$, and so
\[
 j'_2-j'_1 \leq (j'_2-1)-j \leq d_X(y^3_{j},y^3_{j'_2-1}) \leq 2 + \left\lceil\frac{(j'_2-1)-j}4\right\rceil.
\]
We conclude that $j'_2-j'_1\leq 3$. Therefore, $\partial F$ is the union of $\partial F\cap R^1_i$ and a union at most 4 pieces in $\partial F\cap P^3_{ess}$ and two internal arcs in $D$. Thus $\abs{\partial F}<\left(\frac34+\frac{6}{24}\right)\abs{\partial F}$, which is a contradiction.

\medskip

\textbf{Case 2.}  From now on we assume that $e_{1,3}\leq 2$. Let $x,x'$ be the end vertices of $\gamma^1$, with $x\in\gamma^2$, $x'\in\gamma^4$. For each $t$ define vertices $z^{t,t+1}\in P_{ess}^t$, $z^{t+1,t} \in P_{ess}^{t+1}$ as follows. If $(\dagger)$ is present ($(\ddagger)$ when $t=3,4$), define them to be the vertices contained in the tripod, so in this case $z^{2,3}=z^{2,1}$, for example.  If the tripod does not appear and $E_{t,t+1}\neq\emptyset$, define them to be the end vertices of $e^t$.  Finally, if the tripod does not appear and $E_{t,t+1}=\emptyset$, define them both to be $v^t$.  See Figure \ref{fig:ztconfig} for one possible configuration. The proof will be in two steps: first we show $d_X(\set{z^{1,2},z^{2,1}},\set{z^{3,4},z^{4,3}})\leq 8$, and then we show $d_X(z^{s-1,s},\gamma^s)\leq 4$ and $d_X(z^{s,s-1},\gamma^s)\leq 4$ for $s=2,4$. This contradicts $(\ref{eq:longthin})$. 

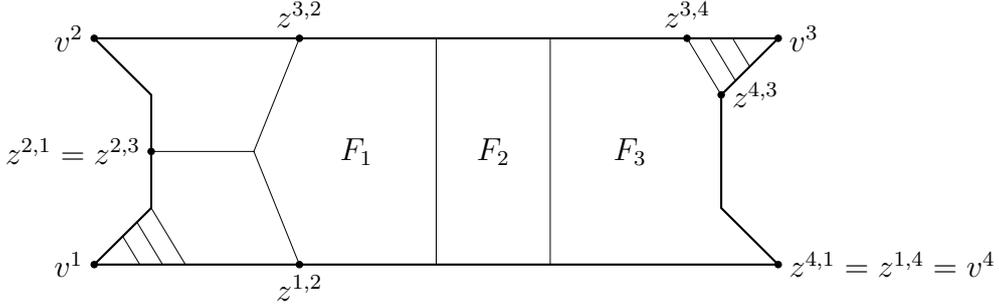
\begin{figure}[H]
  \centering%
\begin{tikzpicture}[yscale=1.5,xscale=1.5, 
vertex/.style={draw,fill,circle,inner sep=0.3mm}]

\draw[thick] (3,1)--(-3,1)--(-2.5,0.5)--(-2.5,-0.5)--(-3,-1)--(3,-1)--(2.5,-0.5)--(2.5,0.5)--(3,1);

\draw[very thin] 	(-2.5,-0.5) -- (-2.2,-1)
					(-2.625,-0.625) -- (-2.4,-1)
					(-2.75,-0.75) -- (-2.6,-1)
					(-2.5,0) -- (-1.6,0) -- (-1.2,1)
					(-1.6,0) -- (-1.2,-1)
					(0,-1) -- (0,1)
					(1,-1) -- (1,1)
					(2.5,0.5) -- (2.2,1)
					(2.625,0.625) -- (2.4,1)
					(2.75,0.75) -- (2.6,1)

					;
					
\path (-0.7,0) node[] {$F_1$};
\path (0.5,0) node[] {$F_2$};
\path (1.7,0) node[] {$F_3$};

\node[vertex]
(x1) at (-3,-1) {};
\path (x1) node[left] {$v^1$};

\node[vertex]
(x2) at (-3,1) {};
\path (x2) node[left] {$v^2$};

\node[vertex]
(x3) at (3,1) {};
\path (x3) node[right] {$v^3$};

\node[vertex]
(x4) at (3,-1) {};
\path (x4) node[right] {$z^{4,1}=z^{1,4}=v^4$};

\node[vertex]
(y1) at (-1.2,1) {};
\path (y1) node[above] {$z^{3,2}$};
\node[vertex]
(y2) at (-1.2,-1) {};
\path (y2) node[below] {$z^{1,2}$};
\node[vertex]
(y3) at (2.2,1) {};
\path (y3) node[above] {$z^{3,4}$};
\node[vertex]
(y4) at (2.5,0.5) {};
\path (y4) node[right] {$z^{4,3}$};
\node[vertex]
(y5) at (-2.5,0) {};
\path (y5) node[left] {$z^{2,1}=z^{2,3}$};

\end{tikzpicture}\\
	\caption{One possibility for the $z^{t,t+1}$ and $z^{t+1,t}$.} \label{fig:ztconfig}
\end{figure}

Label the interior arcs in $E^{2,4}$ by $\alpha_1,\alpha_2$ (when they exist) so that $\alpha_1$ starts closer to $v^1$ along $P^1_{ess}$ than $\alpha^2$.  Let us define a path of length at most $8$ connecting one of $\set{z^{1,2},z^{2,1}}$ to one of $\set{z^{3,4},z^{4,3}}$ in the worst case scenario, which is when $e^{2,4}=2$. Now $z^{3,2}$ and $z^{1,2}$ lie on the boundary of a unique face $F_1\subseteq D$. Since $\partial F_1$ is not simultaneously one of the $R^1_i$ and one of the $R^3_j$, by Lemma \ref{numberoffacesbound}, there are paths in $\partial F$ connecting one of $z^{3,2}$ or $z^{1,2}$ to any vertex of $\alpha_1$ which are unions of at most $6$ pieces in $\partial F_1$.  This implies that $d_X(\set{z^{3,2},z^{1,2}},\alpha_1)\leq 2$. Similarly, there is a path from any vertex in $\alpha_1$ to any vertex in $\alpha_2$ which is a union of at most 7 pieces in $\partial F_2$ (the neighbour of $\partial F_1$ whose boundary contains $\alpha_1$), and there is a path from any vertex in $\alpha_2$ to one of $\set{z^{3,4},z^{1,4}}$ which is a union of at most 6 pieces in $\partial F_3$ (the neighbour of $\partial F_2$ whose boundary contains $\alpha_2$). Combining these observations we see that $d_X(\set{z^{1,2},z^{2,1}},\set{z^{3,4},z^{4,3}})\leq 8$, as required.

For the second step, we will bound $d_X(z^{2,1},\gamma^2)$; all the other bounds can be found in the same way. Note that if the tripod $(\dagger)$ is present or $e^{1,2}\neq 0$, then $d_X(z^{2,1},\gamma^2)\leq 3$ by Lemma \ref{lem:bdgammatoPess}. Hence the only case we have to consider is $z^{2,1}=v^1$. The remainder of the argument closely follows the second half of Lemma \ref{lem:noladder}. Let $\overline D$ be a diagram with boundary $Q_S$ containing $D$ as a subdiagram. Let $F$ be the unique face in $D$ whose boundary contains $v^1$. If there is some $R^2_r$ such that $P^2_{ess} = R^2_r$, then $d_X(\gamma^2,v^1)\leq 2$ by Lemma \ref{disttoPess}.  If this does not happen, but for some $r'$, $P^2_{ess} \cap R^2_{r'}$ has an end vertex $v$ in $\partial F$, then $d_X(\gamma^2,v^1)\leq d_X(\gamma^2,v)+d_X(v,v^1)\leq 2 + \left\lceil\frac54\right\rceil=4$, by Lemmas \ref{disttoPess} and \ref{numberoffacesbound}(i). We may now assume that there is some $r$ such that $P^2_{ess} \cap R^2_{r}$ (strictly) contains $\partial F\cap P^2_{ess}$. We deduce that there is no other face $F'\subset \overline D$ such that $\partial F'\cap R^2_r$ contains an edge using the penultimate paragraph of the proof of Lemma \ref{lem:noladder}. Using the final paragraph of the same proof, we see that either there exist $i,j$ such that $R^1_i=R^3_j$ which is a contradiction, or one of the subpaths $N^s=R^2_r\cap P^2_{ess}\cap P^s_{ess}$ ($s=1,3$) has diameter at most $2$. Therefore,
\[
 d_X(\gamma^2,v^1) \leq d_X(\gamma^2,N^s) + \diam(N^s) + d_X(N^s,v^1) \leq 2+2+1=5,
\]
where the bound on $d_X(\gamma^2,N^s)$ comes from Lemma \ref{disttoPess} and the bound on $d_X(N^s,v^1)$ follows since $N^s$ contains one of the two end vertices of the piece $R^2_r\cap\partial F$.
\end{proof}

\subsection{Acylindricity}

\begin{theorem} \label{thm:acyl} Let $\mathcal P=\fpres{S}{r_1,r_2,\dots}$ be a presentation for a $C'(\frac{1}{24})$ group $G$.  If $G$ is uniformly power-free and $[X]\in\TCG$, then $G\curvearrowright \Cay(G,X)$ is acylindrical.
\end{theorem}

\begin{proof}  
As $[X]\in\TCG$, it follows from Theorem \ref{thm:thinconehyp} that $\Cay(G,X)$ is hyperbolic.

Let $N$ be the uniformly power-free constant, that is, the constant such that no subpath of a relator is labelled by an $N$--th power of a non-trivial word. 

Fix $\varepsilon>0$.  Let $x,y\in \Cay(G,X)$ such that $$d_{ X}(x,y)=3\varepsilon+25 .$$  Let $\gamma^1$ be a geodesic from $x$ to $y$ in $\Cay(G,X)$, and let $g\in G\setminus\{1\}$ be such that 
 \begin{align} \label{acylineq} d_{ X}(x,gx)\leq \varepsilon \quad \textrm{ and } \quad d_{ X}(y,gy)\leq \varepsilon.\end{align}  By \cite[Lemma 2.4]{Os16}, it suffices to show that there is a uniform bound on the number of such $g$. Fix a geodesic $\gamma^1$ in $\Cay(G,X)$ from $y$ to $x$. For each $g$ satisfying $(\ref{acylineq})$, let $\gamma^2$ be a geodesic from $x$ to $gx$ in $\Cay(G,X)$, let  $\gamma^4$ be a geodesic from $gy$ to $y$ in $\Cay(G,X)$, and let $\gamma^3$ be the geodesic $g\gamma^1$ with the opposite orientation.  Define  $Q_{ X}=(\gamma^1,\gamma^2,\gamma^3,\gamma^4)$ to be the quadrangle associated to the choice of $x,y$ and $g$, and let $Q_S=(P^1_{ess},P^2_{ess},P^3_{ess},P^4_{ess})$ be an essential quadrangle associated to $Q_{ X}$ where $P^3_{ess}$ is $gP^1_{ess}$ with the opposite orientation.

 By construction, \[\min\{l_X(\gamma^1),l_X(\gamma^3)\}\geq 3\max\{l_X(\gamma^2),l_X(\gamma^4)\}+25,\]  so $Q_{ X}$ is a ``long, thin quadrangle'' which satisfies ($\ref{eq:longthin}$).  Thus we can apply the classification of quadrangles from Section \ref{subsec:quads}.  By Theorem \ref{thm:quadclass}, either there exists some $i$ and some $j$ such that $R^1_i=R^2_j$, or $P^1_{ess}\cap gP^1_{ess}$ contains a path whose end vertices are at distance at most $\varepsilon +3$ from $x$ and $y$, respectively, in $\Cay(G,X)$.
 \medskip
  
\noindent \textbf{Case 1.} Suppose there exist $i,j$ such that $R^1_i=R^3_j$.  For each $g$ yielding a diagram in this case we have $R^1_i=R^3_j=gR^1_j$, where the last inequality follows because $P^3_{ess}$ is $gP^1_{ess}$ with the opposite orientation. If there are more than $N(3\varepsilon+22)^2$ such $g$, then at least $N+1$ different $g$ satisfy $R^1_i=gR^1_j$ for some fixed $i,j$. If $g_1R^1_{j}=R^1_i=g_2R^1_{j}$, then $g_1^{-1}g_2\in \operatorname{Aut}(C)$, where $C\subseteq \Gamma$ is the cycle corresponding to $R^1_i$.  Hence 
$\abs{\textup{Aut}(C)}\geq N+1$, contradicting the uniformly power-free assumption. 

\medskip
\noindent \textbf{Case 2.} Suppose Case 1 does not hold.  Then for each $g$ yielding such a diagram, $P^1_{ess}\cap gP^1_{ess}$ contains a path whose end vertices are at distance at most $\varepsilon + 3$ from $x$ and $y$, respectively, in $\Cay(G,X)$. Fix a subpath $P$ of $P^1_{ess}$ starting at the last vertex within distance $\varepsilon + 4$ of $x$ and  ending at the first vertex within distance $\varepsilon + 4$ of $y$. It is clear that the end vertices of this path are at least $\varepsilon + 17$ far apart in $\Cay(G,X)$, and that $P$ is a subpath of $gP^1_{ess}$ for all $g$ under consideration.  The element $g$ is uniquely determined  by the length of the subpath of $gP^1_{ess}$ connecting $gx$ to the starting point of $P$ on $P^1_{ess}$.  Suppose there are more than $N(3\varepsilon+22)$ such different elements $g$.  Then there is some $i$ and $N+1$ different starting points $z_0,\ldots,z_N$ of $P$ on $R^1_i\cap P^1_{ess}$ such that one of the following occurs. \begin{enumerate}[(i)]
\item The label of the subpath of $P^1_i$ starting at $z_0$ is equal to an initial subword of the label of $P$
\item The label of $P$ is equal to an initial subword of the label of the subpath $P^1_i$ starting at $z_0$. 
\end{enumerate}

Let us show that the second option cannot happen. If it does, then there is some $g$ and some $i$ such that $P\subseteq gP^1_i\cap gP^1_{ess}$. Since we are not in case 1, $gR^1_i\neq R^1_j$ for any $j$, and so $P$ is contained in a union of at most $5$ pieces $gR^1_i\cap R^1_j$.  Thus its end vertices are at distance at most $ \lceil \frac54 \rceil =2$ in $\hat X$, which is a contradiction.

Therefore, there exist $l\neq m$ such that $0 < r=d_S(z_l,z_m)\leq \frac1N \abs{P}$ and the label of $P$ has an initial subword which is the $N$--th power of the label of its initial subpath of length $r$. This contradicts the uniformly power free assumption.

\end{proof}

\begin{proof}[Proof of Theorem \ref{thm:hypandacyl}]
Let $\mathcal P$ be a $C'(\frac{1}{24})$  presentation for a group $G$.  Then it follows from Theorem \ref{thm:thinconehyp} that $\TCG\subset\mathcal H(G)$. If $\mathcal P$ is additionally uniformly power-free, then $\TCG\subset \mathcal{AH}(G)$ by Theorem \ref{thm:acyl}.
\end{proof}

\section{Thin cones and $\mathcal{(A)H}$--inaccessibility}\label{sec:largest}

The goal of this section is to prove Theorems \ref{thmi:TCG} and \ref{thmi:notwklyaccess}, which will be done in subsections \ref{sec:subsec1} and \ref{sec:subsec2}, respectively.

\subsection{The structure of the subposet $\TCG$}\label{sec:subsec1}

Throughout this section $\mathcal{P}=\fpres{S}{R}$ is a $C'(\frac{1}{24})$ presentation, each element of $r_i$ is cyclically reduced in $F(S)$ and we enumerate $R=\set{r_1,r_2,\ldots}$. Define $\overline R$ to be the set of all cyclically reduced conjugates of the $r_i$ and their inverses. For each $i$ we define $C_i$ to be a cyclic graph whose label is $r_i$. We define $L$ to be the set of all initial subwords of elements of $\overline R$ and $P^k$ to be the set of words which are a product of at most $k$ pieces.

Recall that given $Y\subset F(S)$, $C_i^Y$ is the graph obtained from $C_i$ by adding an edge connecting any two vertices in $C_i$ such that there is a path between them whose label is in $Y$. We use the shorthand $\overline{C_i}=C_i^{P^4}$.

Let us recall the construction of the laced cone from Example \ref{def:lacedcone1}.

\begin{definition} For each $i$, fix a vertex $x_i$ in $\overline{C_i}$, and define $\mathcal{P}_i$ to be the set of all paths in $C_i$ which connect two points $y,z$ such that $d_{\overline{C_i}}(x_i,y)=d_{\overline{C_i}}(x_i,z)$. Set
\[
 LC((x_i)_i) = S\cup P^4 \cup \bigcup_{i\geq 1} \setcon{\Lab(P)}{P\in\mathcal P_i} \subseteq L.
\]
We call $LC((x_i)_i)$ the \textbf{laced cone based at $(x_i)_i$} (cf. Figure \ref{fig:lacedcone}).
\end{definition}
It is clear that $[LC((x_i)_i)]\in\TCG$. We now begin our study of the poset $\TCG$.

\begin{lemma}\label{lem:TCG1}
$|\TCG|=1$ if and only if each $C_i$ is a union of a uniformly bounded number of pieces (or equivalently $P^k=L$ for some $k$).
\end{lemma}

\begin{proof} It is clear that if each $C_i$ is a union of a uniformly bounded number of pieces $M$, then for any $X\subseteq P^4$, the $C_i^X$ have uniformly bounded diameter $M/4$, and therefore $X$ is equivalent in $\TCG$ to the smallest thin cone, $L$. 

Now suppose this is not the case. For each $j\in \N$, let $w_j$ be the label of a subpath $P_j$ in some $C_{i_j}$ of length at most $\frac{\abs{r_{i_j}}}{2}$ such that $w\not\in P^{4j}$. If this cannot be done for some $j$, then $L\subseteq P^{8j}$. Without loss of generality, we may assume the map $j\mapsto i_j$ is injective. Let $x_j$ and $w_j$ be the end vertices of $P_j$, and consider any laced coned-off graph $\Gamma=\Cay(G,X((x_i)_i))$. By Lemma \ref{lem:qcxity}, the cones $C_{i_j}^{X((x_i)_i)}$ isometrically embed into $\Gamma$, so
$d_{X((x_i)_i)}(1,w_j)\geq j$ for all $j$, while $d_L(1,w_j)=1$. Therefore $[X((x_i)_i)]\not\preceq [L]$.
\end{proof}

We next show that $\TCG$ has a largest element if and only if it has exactly one element.

\begin{lemma}\label{lem:nothingabove} If $|\TCG|\neq 1$ then there are two elements $[X],[Y]\in \TCG$ such that no $[Z]\in \TCG$ satisfies both $[X]\preceq[Z]$ and $[Y]\preceq[Z]$.
\end{lemma}
\begin{proof} Applying Lemma \ref{lem:TCG1}, if $|\TCG|\neq 1$, then for each $j\in \N$ there is some $C_{i(j)}$ whose boundary word cannot be written as a product of fewer than $16j$ relators. Thus $d_j=\diam(\overline{C_{i(j)}})\geq 2j$. Choose $x_{i(j)},y_{i(j)}\in C_{i(j)}$ satisfying
\[
 d_{\overline{C_{i(j)}}}(x_{i(j)},y_{i(j)})=\left\lfloor \frac{d_j}{2}\right\rfloor,
\]
 and fix a vertex $x_i=y_i\in C_i$ for any $i$ which is not equal to one of the $i(j)$. Consider the laced cones $X=LC((x_i)_i)$ and $Y=LC((y_i)_i)$.

Suppose for a contradiction that there is some $[Z]\in\TCG$ such that $[X]\preceq[Z]$ and $[Y]\preceq[Z]$. For any representative $Z$ of $[Z]$ there is a constant $K$ such that
\begin{equation}\label{eqn:biLip} d_Z(a,b)\geq \frac1K\max\set{d_{X}(a,b),d_Y(a,b)}-K
\end{equation}
holds for all $a,b\in G$.

Now suppose $a,b$ are contained in a common cone $C^Z$. If $d_{\overline{C_{i(j)}}}(a,b)\leq \lfloor \frac{d_j}{2}\rfloor$ then by Lemma \ref{lem:qcxity} either $d_{X}(a,b)=d_{\overline{C_{i(j)}}}(a,b)$ or $d_{Y}(a,b)=d_{\overline{C_{i(j)}}}(a,b)$. If $d_{\overline{C_{i(j)}}}(a,b)> \lfloor \frac{d_j}{2}\rfloor$, then either $d_{X}(a,b)> \frac12\lfloor \frac{d_j}{2}\rfloor$ or $d_{Y}(a,b)> \frac12\lfloor \frac{d_j}{2}\rfloor$. Since in either case $d_{\overline{C_{i(j)}}}(a,b)\leq d_j=\diam(\overline{C_{i(j)}})$, it follows from Lemma \ref{lem:qcxity} that $\max\{d_{X}(a,b),d_Y(a,b)\}\geq \frac12d_{\overline{C_{i(j)}}}(a,b)$.  Combining this with (\ref{eqn:biLip}), we have
\[
 d_{\overline{C_{i(j)}}}(a,b) \geq d_Z(a,b) \geq \frac{1}{2K} d_{\overline{C_{i(j)}}}(a,b).
\]
Hence the $\overline{C_{i(j)}}$ are uniformly biLipschitzly embedded in $\Cay(G,Z)$, so if $Z$ is a thin cone in $[Z]$, by Lemma \ref{lem:qcxity} $\Cay(G,Z)$ contains biLipschitzly embedded cycles of arbitrary diameter.  However, this implies that $\Cay(G,Z)$  is not hyperbolic and hence by Theorem \ref{thm:thinconehyp}, $Z$ is not equivalent to any element of $\TCG$, contradicting our assumption.
\end{proof}

\begin{remark} It is not possible to infer from Lemma \ref{lem:nothingabove} that $G$ is not $\mathcal{H}$--accessible since there are cobounded actions $G\curvearrowright Z$ such that $\sigma([G\curvearrowright Z])$ does not contain any thin cones. In the proof of the lemma we need to assume that $Z$ is a thin cone in order to apply Lemma \ref{lem:qcxity}. Later we will give a stronger hypothesis from which we can deduce that a group is not $\mathcal{H}$--accessible.
\end{remark}

We next show that for every element of $\TCG\setminus\set{[L]}$ there is another element $[Y]\in\TCG$ which is incomparable to $[X]$. Recall that given $X\subseteq F(S)$, $X_i$ is the set of all elements of $X$ which are the initial subword of a cylically reduced conjugate of either $r_i$ or its inverse.

\begin{lemma}\label{lem:incomparable} For all  $[X]\in\TCG\setminus\set{[L]}$ there exists a $[Y]\in\TCG$ such that $[X]\not\preceq[Y]$ and $[Y]\not\preceq[X]$.
\end{lemma}
\begin{proof} Let $[X]$ be any element of $\TCG$ other than $[L]$, the smallest element.  By Lemma \ref{lem:nothingabove}, there exists $[Z]\in\TCG$ such that $[Z]\not\preceq [X]$. Let $Z$ be a thin cone in $[Z]$. Since $[X]\not\preceq [L]$, there exists an infinite set $\{i(j)\mid j\in\mathbb N\}$ and a sequence of elements $w_{i(j)}\in L_{i(j)}$ such that $\sup_{j}|w_{i(j)}|_X=\infty$.  Similarly, since $[Z]\not\preceq[X]$, there exists an infinite set $\{i'(k)\mid k\in\mathbb N\}$ and a sequence of elements $w'_{i'(k)}\in X_{i'(k)}$ such that $\sup_{k}|w'_{i'(k)}|_Z=\infty$.  Let $I$ be an infinite subset of $\{i'(k)\mid k\in\mathbb Z\}$ such that $\{i(j)\mid j\in \mathbb Z\}\setminus I$ is infinite.  Define 
\[Y_l=\begin{cases} L_l& \textrm{ if } l\not\in I \\ Z_{i'(k)} & \textrm{ if } l=i'(k)\in I \textrm{ for some $k$.}\end{cases}\] 
Let $Y=\bigcup_l Y_l$, it is clear that $Y$ is a thin cone.  Now $[X]\not\preceq [Y]$, as for all $l$ such that $l=i(j)\not\in I$, we have $w_{i(j)}\in L_{i(j)}=Y_{i(j)}$, which implies that $\sup_j|w_{i(j)}|_Y=1$, while $\sup_j|w_{i(j)}|_X=\infty$.  Similarly, $[Y]\not\preceq [X]$ since for all $l$ such that $l=i'(k)\in I$, we have $w'_{i'(k)}\in X_{i'(k)}$, which implies that $\sup_k|w'_{i'(k)}|_X=1$, while $\sup_k|w'_{i'(k)}|_Y=|w'_{i'(k)}|_Z=\infty$.
\end{proof}

Recall that $\mathcal P(\omega)/Fin$ is the poset of equivalence classes of subsets of $\mathbb N$, where two subsets $A,B\subseteq \mathbb N$ are equivalent if $|A\triangle B|<\infty$ and $A\leq B$ if $|A\setminus B|<\infty$. Our next goal is to show that $\TCG$ is large.

\begin{proposition} \label{prop:PFin}
For any distinct elements $[X^1],[X^2]\in\TCG$ such that $[X^1]\preceq [X^2]$, there is an embedding of posets $\phi\colon\mathcal P(\omega)/Fin\hookrightarrow \TCG$ such that for each $[A]\in\mathcal P(\omega)/Fin$, $\phi([\emptyset])=[X^1]\preceq \phi([A]) \preceq[X^2]=\phi([\N])$.
\end{proposition}

\begin{proof}
Let $\mathcal P=\fpres{S}{r_1,\ldots}$ be a $C'(\frac{1}{24})$ presentation.
Let $[X^1],[X^2]\in\TCG$ be two distinct elements satisfying $[X^1]\preceq[X^2]$. We have $X^1\sim X^1\cup X^2$, so by possibly changing representatives of the equivalence classes, we may assume without loss of generality that $X^2\subset X^1$. Since $X^2\not\sim X^1$, 
there is an infinite set $I=\{n(i)\}\subset \mathbb N$ and 
a sequence of elements $w_{n(i)}\in X^2_{n(i)}$ such that \begin{equation}\label{eqn:supI}\sup_{n(i)\in I}|w_{n(i)}|_{X^1}=\infty.\end{equation} 

Given a subset $A\subseteq \mathbb N$,  let $I_A=\{n(i)\in I\mid i\in A\}\subset I.$  Define
\[ X^A_j=\begin{cases}  X^2_j & \textrm{ if } j\not\in I \textrm{ or } j\in I_A \\ X^1_j & \textrm{ if } j\in I\setminus I_A\end{cases}.\]  

Let $X^A=\sqcup_j X^A_j$.  By construction, $[X^A]\in \TCG$.  Consider the map \[\phi\colon \PFin \to \TCG\] defined by \[[A]\mapsto[X^A].\]

We first show that $\phi$ is well-defined. Suppose $A,B\subseteq\mathbb N$ are equivalent in $\PFin$, i.e, $|A\triangle B|<\infty$.  It suffices to consider the case $B=A\cup\{b\}$. Then $X^A$ and $X^B$ differ only in $X^A_b$ and $X^B_b$.  Recall that elements of $X^A_b$ and $X^B_b$ are labels of the edges added to the cycle $C_b$.  Since $C_b$ has only finitely many vertices, it follows that $|X^A\triangle X^B|<\infty$, and so $[X^A]\sim [X^B]$.

We next show $\phi$ is injective. Suppose $A,B\subseteq\mathbb N$ are not equivalent in $\PFin$, i.e., $|A\triangle B|=\infty$.  We may assume without loss of generality that $B\setminus A$ is infinite.  We will show that $X^A$ and $X^B$ are not equivalent.   
Since $B\setminus A$ is infinite, there is an infinite subsequence $(n(i,k))$ of $(n(i))$ such that $w_{n(i,k)}\in X^A_{n(i,k)}= X^1_{n(i,k)}$ while $ X^B_{n(i,k)}= X^2_{n(i,k)}$.  Since the sequence $(n(i,k))$ is infinite, we must have $n(i,k)\to \infty$ as $k\to \infty$, and so by (\ref{eqn:supI}) we have \[\sup_k|w_{n(i,k)}|_{X^B}=\sup _k|w_{n(i,k)}|_{X^2}= \infty.\]  Therefore $X^A\not\sim X^B$.

Finally, we show $\phi$ is order-preserving. Suppose $[A],[B]\in\PFin$ satisfy $[A]\leq [B]$.  By changing representatives, we may assume that $A\subseteq B$. It then follows from the definition that $ X^A \supseteq X^B$, and so $[X^A]\preceq[X^B]$.

Finally, by construction $X^1\supseteq X^A\supseteq X^2$  for all $[A]\in\PFin$, and so $[X^1]\preceq [X^A]\preceq[X^2]$ for all $[A]\in\PFin$.  Moreover, by construction $\phi(\emptyset)= [X^\emptyset]=[X^1]$ and $\phi(\mathbb N)=[X^\mathbb N]=[X^2]$.
\end{proof}

\begin{proposition} \label{prop:antichain}
Every $[X]\in\TCG$ which is not the smallest element is contained in an uncountable chain and in an uncountable antichain in $\TCG$.
\end{proposition}

\begin{proof}
Recall that $[L]\in\TCG$ is the smallest element, and let $[X]\in\TCG$ such that $X\not\sim L$.   By Proposition \ref{prop:PFin}, there is an embedding of posets $\phi\colon\PFin\hookrightarrow \TCG$ such that $[L]\preceq\phi([A])\preceq [X]$ for all $[A]\in\PFin$.  Since $\PFin$ contains uncountable chains, $[X]$ is contained in an uncountable chain.

To show that $[X]$ is contained in an uncountable antichain, first choose some $[Y]\in\TCG$ such that $[X]$ and $[Y]$ are incomparable, which exists by Lemma \ref{lem:incomparable}.  Since $[X]$ and $[Y]$ are incomparable, there exist subsequences $(n(i))$ and $(n'(j))$ of the natural numbers and words $w_{n(i)}\in X_{n(i)}$ and $w'_{n'(j)}\in Y_{n'(j)}$ such that \begin{equation} \label{eqn:supY}\sup_i|w_{n(i)}|_Y=\infty\end{equation} and \begin{equation}\label{eqn:supX}\sup_j|w_{n'(j)}|_X=\infty.\end{equation}

There are two cases to consider.  If $|\{n(i)\}\triangle \{n'(j)\}|=\infty$, then by passing to subsequences we may assume that $\{n(i)\}\cap \{n'(j)\}=\emptyset$.  In this case, given $A\subseteq\mathbb N$, let $I_A=\{n(i)\mid i\in A\}$ and $J_A=\{n'(j)\mid j\in A\}$.  If $|\{n(i)\}\triangle \{n'(j)\}|<\infty$, then by passing to subsequences, we may assume that $\{n(i)\}= \{n'(j)\}$.  In this case, given $A\subseteq\mathbb N$, let $I_A=\{n(i)\mid i\in A\}$ and $J_A=\{n(i)\mid i\not\in A\}$.  

In either case, define 
\[ W^A_k=\begin{cases}  X_k & \textrm{ if $k\in I_A$} \\
 Y_k& \textrm{ if $k\in J_A$} \\
 L_k & \textrm{ else}
\end{cases}.\]  
Let $W^A=\sqcup_k W_A^k$.  Define a map \[\phi\colon \PFin \to \TCG\] by \[[A]\mapsto [W^A].\]  By a similar argument as in the proof of Proposition \ref{prop:PFin}, this map is well-defined.

We now show that if $[A],[B]\in\PFin$ are incomparable, then $[W^A]$ and $[W^B]$ are incomparable in $\TCG$.  Suppose $A,B\subseteq\mathbb N$ satisfy $|A\setminus B|=\infty$ and $|B\setminus A|=\infty$.      

If $|\{n(i)\}\triangle\{n'(j)\}|=\infty$, then for $i\in A\setminus B$, consider the subsequence of words $w_{n(i)}\in W^A_{n(i)}=X_{n(i)}$ and for $j\in B\setminus A$ consider the sibsequence $w'_{n'(j)}\in W^B_{n'(j)}=Y_{n'(j)}$.   Since $A\setminus B$ and $B\setminus A$ are both infinite sets, these are both infinite sequences of words.  Applying (\ref{eqn:supY}) yields  \[\sup_i|w_{n(i)}|_{W^B}=\sup_i|w_{n(i)}|_Y=\infty,\] so $W^B\not\preceq W^A$.  Similarly, applying (\ref{eqn:supX}) yields \[\sup_j|w'_{n'(j)}|_{W^A}=\sup_j|w'_{n'(j)}|_X=\infty,\] and so $W^A\not\preceq W^B$.  Therefore, $[W^A]$ and $[W^B]$ are incomparable.  

If $|\{n(i)\}\triangle\{n'(j)\}|<\infty$, then since $B\setminus A\subseteq A^c$, a similar argument shows that $[W^A]$ and $[W^B]$ are incomparable.

Suppose $[A]\in\PFin$ is such that $A\not\sim \emptyset$ and $A\not\sim\mathbb N$.  Then $A^c$ is an infinite set, and a similar argument shows that  $[X]$ and $[W^A]$ are incomparable and $[Y]$ and $[W^A]$ are incomparable.

Therefore, for any antichain  $\{[A_\alpha]\}$  in $\PFin$, $\{[W^{A_\alpha}]\}$ is an antichain in $\TCG$ which can be extended to include $[X]$ and $[Y]$.  Since $\PFin$ contains uncountable antichains, the result follows.

\end{proof}

\begin{proof}[Proof of Theorem \ref{thmi:TCG}]
If $G$ is uniformly power-free then Theorem \ref{thm:acyl} implies that $\TCG\subseteq \mathcal{AH}(G)$.

Part (i) is Lemma \ref{lem:TCG1}, while part (ii) follows  from Propositions \ref{prop:PFin} and \ref{prop:antichain}.

\end{proof}

\subsection{$\mathcal H$-- and $\mathcal{AH}$--inaccessible groups}\label{sec:subsec2}
In this section, we construct examples of groups that are neither $\mathcal H$-- nor $\mathcal{AH}$--accessible.  We in fact prove an even stronger result: there is no largest (not necessarily cobounded) action under the partial ordering on a hyperbolic space.  Moreover, our examples are all groups which admit universal acylindrical actions.  Our main tool is to use thin cones to construct many ``sufficiently different" actions of a group on hyperbolic spaces.  We then show that if the group admits an action that is larger than all of these different actions in the partial ordering, it cannot be on a hyperbolic space.

\begin{theorem}\label{thm:notMA} Let $\mathcal{P}=\fpres{S}{R}$ be a $C'(\frac{1}{14})$ presentation where each $r\in R$ is cyclically reduced. Suppose that for every $n$ there is some $r_n\in R$, which is not a proper power, satisfying
\[ \lim_{n\to\infty}\abs{r_n}_S=+\infty \quad \textup{and} \quad \inf_{n\to\infty}\frac{\abs{r_n}_S^\frac12}{p(r_n)\log_2\abs{r_n}_S}>0,
\]
where $p(r_n)$ is the length of the longest piece in $r_n$. Then $G=F(S)/\ngen{R}$ does not admit a largest action on a hyperbolic space.
\end{theorem}

The strategy of the proof is to use the two laced cones considered in Lemma \ref{lem:nothingabove} to prove that if $G\curvearrowright Z$ is an action such that for any $[X]\in\TCG$,
\[
 d_{{X}}(g,h)\leq Kd_Z(g,h) + C,
\] 
then $Y$ does not have exponential divergence, and hence it is not hyperbolic.

\begin{proof} Fix $z\in Z$, and for each $n$ let $C_n$ denote the labelled cyclic subgraph of $\Cay(G,S)$ which contains the vertex $1_G$ and has label $r_n$ when read from $1_G$. We denote by $\overline{C_n}$ the induced subgraph of $\Cay(G,S\cup P^4)$ with the same vertex set as $C_n$. Using the proof of Lemma \ref{lem:nothingabove} we find $[X],[Y]\in\TCG$ and $K_1\geq 1$ such that for all $a,b$ in a common $C_n$ we have
\begin{equation}\label{ZvsolCndist}
 \frac{1}{2K_1}d_{\overline{C_n}}(a,b)-K_1 \leq d_Z(a.z,b.z).
\end{equation}
Moreover, since the orbit map $g\mapsto g.z$ is $K_2$--Lipschitz, we have
\begin{equation}\label{ZvsCndist}
d_Z(a.z,b.z)\leq  K_2d_{C_n}(a,b)+K_2.
\end{equation}
Set $K=\max\set{K_1,K_2}$.

Define $f:\N\to\R$ by $f(n)=n^{\frac12}/\log_2(n)$. By hypothesis, there is an infinite subset $I\subseteq\N$ such that for all $n\in I$
\[ \frac{f(\abs{r_n}_S)}{p(r_n)} \geq \eps>0.
\]
Therefore, for all pairs of vertices $a,b$ in $C_n$, we have
\[
d_{\overline{C_n}}(a,b) \geq \frac{d_{C_n}(a,b)}{4p(r_n)} \geq \frac{\epsilon d_{C_n}(a,b)}{4f(\abs{r_n}_S)}.
\]
For each $n$ choose $a_n,b_n$ with $d_{C_n}(a_n,b_n)$ maximal (in particular, it is greater than $\frac13\abs{r_n}_S$). From the above equation we see that 
\begin{equation}\label{Cnbardist}
 d_{\overline{C_n}}(a_n,b_n)\geq \frac{\epsilon}{12} \frac{\abs{r_n}_S}{f(\abs{r_n}_S)}.
\end{equation} 
Let $P_n^1$ and $P_n^2$ be the two different embedded paths in $C_n$ from $a_n$ to $b_n$, and for $k=1,2$, let $q_n^k$ be a path in $Z$ obtained by connecting the images of consecutive vertices of $P_n^k$ under the orbit map $g\mapsto g.z$ by geodesics of length at most $2K$. This is possible by (\ref{ZvsCndist}). The length of the path $q_n^k$ is at most $2K\abs{r_n}_S$. Let $\ell$ be a geodesic connecting $a_n.z$ and $b_n.z$ in $Z$ and let $m_n$ be the midpoint of $\ell$. By \cite[Proposition 3.H.1.6]{BH99}, since $Z$ is $\delta$-hyperbolic, $d_Z(m_n,q_n^k)\leq \delta \log_2(2K\abs{r_n}_S)+1$, therefore there exist points $c_n,d_n\in C_n$ such that
\begin{equation}\label{vclosepoints}
 (2K)^{-1}d_{\overline{C_n}}(c_n,d_n)\leq d_Z(c_n.z,d_n.z) \leq 2\delta\log_2(2K\abs{r_n}_S)+2+4K.
\end{equation}
Moreover, for all $n$ sufficiently large, $\delta \log_2(2K\abs{r_n}_S) + 1 + 2K\leq \frac{\epsilon}{96K}\frac{\abs{r_n}_S}{f(\abs{r_n}_S)}$.  We claim that
\begin{equation}\label{vfarpoints}
 d_Z(\set{c_n.z,d_n.z},\set{a_n.z,b_n.z}) \geq \left(\frac{\eps}{96K}\right)\frac{\abs{r_n}_S}{f(\abs{r_n}_S)}.
\end{equation}
To see this, suppose that that $d_Z(c_n.z,a_n.z)<\left(\frac{\eps}{96K}\right)\frac{\abs{r_n}_S}{f(\abs{r_n}_S)}$; the other possibilities are similar.  Then 
\[d_Z(a_n.z,m_n)=d_Z(a_n.z,c_n.z)+d_Z(c_n.z,m_n)< 2\left(\frac{\eps}{96K}\right)\frac{\abs{r_n}_S}{f(\abs{r_n}_S)},\]
and so 
$d_Z(a_n.z,b_n.z)< 4\left(\frac{\eps}{96K}\right)\frac{\abs{r_n}_S}{f(\abs{r_n}_S)}$. However, by (\ref{ZvsolCndist}), this implies that
\[d_{\overline{C_n}}(a_n,b_n)< 8K\left(\frac{\eps}{96K}\right)\frac{\abs{r_n}_S}{f(\abs{r_n}_S)} = \left(\frac{\eps}{12}\right)\frac{\abs{r_n}_S}{f(\abs{r_n}_S)},\] 
which contradicts (\ref{Cnbardist}).  Thus (\ref{vfarpoints}) holds.

From (\ref{vclosepoints}) we see that there is a path in $C_n$ from $c_n$ to $d_n$ which is a union of at most $16K(\delta\log_2(2K\abs{r_n}_S)+1+2K)$ pieces. Since this path contains either $a_n$ or $b_n$, and each piece has length at most $p(r_n)\leq \epsilon^{-1}f(\abs{r_n}_S)$ in $C_n$,  we see that
\begin{align*}
 16\delta\epsilon^{-1} K\left(\log_2(2K\abs{r_n}_S)+1+2K\right)f(\abs{r_n}_S) &\geq d_{C_n}(c_n,d_n) \\
& \geq d_{C_n}(\set{c_n,d_n},\set{a_n,b_n}) \\
&\geq \frac{\epsilon\abs{r_n}_S}{96K^2f(\abs{r_n}_S)},
\end{align*}
where the final inequality follows from (\ref{vfarpoints}).
Thus there exists some $M\geq 1$, such that for all $n$ sufficiently large we have
\[
 Mf(\abs{r_n}_S)^2\log_2\abs{r_n}_S \geq \frac{1}{M} \abs{r_n}_S.
\]
However, this contradicts the definition of $f$, since for all sufficiently large $n\in I$,
\[
 \frac{\abs{r_n}_S}{f(\abs{r_n}_S)^2\log_2\abs{r_n}_S} \geq \epsilon^2\log_2\abs{r_n}_S>M^2. \qedhere
\]
\end{proof}

We conclude by building the first examples of universally acylindrically hyperbolic groups which are neither weakly $\mathcal H$-- nor weakly $\mathcal {AH}$--accessible.

\begin{theorem} There are uncountably many quasi-isometry classes of finitely generated acylindrically hyperbolic groups admitting a universal acylindrical action but no largest action on a hyperbolic space.
\end{theorem}
\begin{proof}
By Theorems \ref{thm:acyl} and \ref{thm:notMA} it suffices to define a uniformly power free $C'(\frac{1}{24})$ presentation $\mathcal P=\fpres{a,b,c}{r_1,\ldots}$ such that $\abs{r_n}_S\to\infty$ and $\frac{\abs{r_n}_S^\frac12}{p(r_n)\log_2\abs{r_n}_S} \geq 1$.

The number of cube-free binary words of length $n$ is at least $2^{\frac{n}{9}+1}$ by \cite[Theorem 7]{Bran-cubefree}.
For each $n\geq 6$ enumerate $2^n$ different cube-free words of length $9n$ in $\set{a,b}^{9n}$ as $w_n^1, \ldots w_n^{2^n}$ and define
\[
 r'_n = \Pi_{i=1}^{2^n} c w_n^i \in F(a,b,c).
\]
We claim that there is some $n_0$ such that $\mathcal P=\fpres{a,b,c}{r'_{n_0},r'_{n_0+1}\ldots}$ suffices. First, each $r'_n$ is cube-free. If $w^3$ is a subword of $r'_n$, then $w$ must contain a $c$ and have length in $F(a,b,c)$ which is a multiple of $9n+1$, but this implies that two of the $w_n^i$ are equal, which is a contradiction.

Secondly, the length of the word $r'_n$ is at least $9n2^n$, and any piece in $r'_n$ is a subword of some $w_n^icw_n^{i+1}$ (with $i$ considered modulo $2^n$), and so has length at most $18n+1$. Since $n\geq 6$, $9n2^n\geq 576n > 384n+24=24(18n+1)$.  Thus $\fpres{a,b,c}{r'_6,r_7',\ldots}$ satisfies $C'(\frac{1}{24})$. Moreover, there exists some $n_0\geq 6$ such that for all $n\geq n_0$,
\[
\frac{\abs{r'_n}_S^\frac12}{p(r'_n)\log_2\abs{r'_n}_S} \geq \frac{(9n)^\frac12 2^{\frac{n}{2}}}{(18n+1)(n+\log_2(9n+1))} \geq 1.
\]
Now take $r_m=r'_{n_0+m}$ for all $m\geq 1$. 
Taking various sparse infinite subcollections $\mathcal R$ of $\setcon{r_n}{n\geq n_0}$ and using Bowditch's taut loop spectrum \cite{BowditchUncQI} as an invariant we obtain uncountably many quasi--isometry classes of groups satisfying the hypotheses of the theorem.
\end{proof}

\bibliographystyle{alpha}
\bibliography{Largest_SC}

\end{document}